\documentclass[a4paper]{article}

\usepackage[english]{babel}

\usepackage[utf8]{inputenc}
\usepackage[utf8]{inputenc}
\usepackage{amsmath}
\usepackage{graphicx}
\usepackage{amssymb}
\usepackage{amsthm}
\usepackage{tikz-cd}
\usepackage{mathrsfs}
\usepackage[colorinlistoftodos]{todonotes}
\usepackage{enumitem}
\usepackage{yfonts}
\usepackage{ dsfont }
\usepackage{MnSymbol}
\usepackage{slashed}

\title{Valuative independence and metric SYZ conjecture}

\author{Yang Li}

\date{\today}
\newtheorem{thm}{Theorem}[section]
\newtheorem{lem}[thm]{Lemma}

\theoremstyle{definition}

\newtheorem{conj}[thm]{Conjecture}
\newtheorem{cor}[thm]{Corollary}

\newtheorem{rmk}[thm]{Remark}
\newtheorem{prop}[thm]{Proposition}
\newtheorem{Def}[thm]{Definition}

\newtheorem*{Notation}{Notation}

\newtheorem*{Acknowledgement}{Acknowledgement}

\newcommand{\cf}{\emph{cf.} }

\newcommand{\R}{\mathbb{R}}
\newcommand{\C}{\mathbb{C}}
\newcommand{\Z}{\mathbb{Z}}
\newcommand{\N}{\mathbb{N}}
\newcommand{\Q}{\mathbb{Q}}

\newcommand{\norm}[1]{\left\lVert#1\right\rVert}

\def\XXint#1#2#3{{\setbox0=\hbox{$#1{#2#3}{\int}$ }
		\vcenter{\hbox{$#2#3$ }}\kern-.6\wd0}}

\begin{document}
	\maketitle

	\begin{abstract}
		Given a polarised maximal degeneration of compact Calabi-Yau manifolds, assuming there exists  a canonical basis of the section ring for the polarisation line bundle, satisfying the valuative independence condition, we will prove the metric SYZ conjecture.
	\end{abstract}

	\section{Introduction}

		Let $\pi: X\to \mathbb{D}^*$ be a \emph{meromorphic degeneration} of compact Calabi-Yau (CY) manifolds over the punctured disc with coordinate $t$, equipped with a relatively ample \emph{polarisation} line bundle $L\to X$.   We assume there is a nowhere vanishing holomorphic volume form $\Omega\in H^0(X, K_X)$, which induces the fibrewise holomorphic volume form $\Omega_t$ by $\Omega=\Omega_t\wedge dt$. The complexity of the degeneration family  $\pi: X\to \mathbb{D}^*$ is measured by a birational invariant called the \emph{essential skeleton} $Sk(X)$, which is a simplicial complex of dimension $0\leq \dim Sk(X)\leq n$. The case of $\dim Sk(X)=n$ is known as  \emph{maximal degeneration}.

	We consider the CY metric $\omega_{CY,t}$ on $X_t$ in the K\"ahler class $\frac{1}{|\log |t||} c_1(L)$, defined by the complex Monge-Amp\`ere equation
	\[
	\omega_{CY,t}^n= \frac{ (L^n) }{ |\log |t||^n  } \mu_t,
	\]
	where $\mu_t$ is the normalised CY probability measure
	\[
	\mu_t=  \frac{  \Omega_t\wedge \overline{\Omega}_t}{    \int_{X_t}  \Omega_t\wedge \overline{\Omega}_t    }.
	\]
	A central problem in CY geometry is the Strominger-Yau-Zaslow (SYZ) conjecture, which concerns the limiting behaviour of the CY metric  $\omega_{CY,t}$ as $t\to 0$. 
	
	\begin{conj}
		(Metric SYZ conjecture) Let $\pi: (X,L)\to \mathbb{D}^*$ be a polarised maximal degeneration of  compact CY manifolds. Given any $0<\delta\ll 1$, for sufficiently small $t$ depending on $\delta$, there exists a special Lagrangian $T^n$-fibration with respect to the Calabi-Yau structure $(X_t, \omega_{CY,t}, \Omega_t)$, on an open subset of $X_t$ whose normalised CY measure is at least $1-\delta$.
	\end{conj}

	The metric SYZ conjecture is currently known for Abelian varieties, certain K3 surfaces, and  a large class of hypersurface examples inside toric Fano manifolds \cite{GrossWilson, LiFermat, LiFano, Hultgren, Hultgren2, PilleSchneider, GotoOdaka}. A more general approach is based on \emph{non-archimedean (NA) pluripotential theory} \cite{Boucksom,Boucksomsurvey, LiNA}. Some key features are as follows (see also Section \ref{sec:NA} for some quick recap):

	\begin{enumerate}
		\item    Given the polarised
		meromorphic degeneration family of compact CY manifolds $\pi: X\to \mathbb{D}^*$, we can take the base change to a degeneration family over the formal disc $X_K\to \text{Spec}(K)$ where $K=\C(\!(t)\!)$, and then canonically
		associate a NA object called the \emph{Berkovich space} $X_K^{an}$.  There is a \emph{hybrid topology} on $X\sqcup X_K^{an}$, so that $X_K^{an}$ can be viewed as a tropical limit of $X_t$ as $t\to 0$.

		\item 	
		
		The essential skeleton $Sk(X)$ is canonically embedded as a subset of $X_K^{an}$.		There is a canonical Lebesgue measure $\mu_0$ supported on $Sk(X) \subset X_K^{an}$, which is the natural weak limit of normalised CY measures $\mu_t$
		on $X_t$ in the hybrid topology. The central result of NA pluripotential theory due to Boucksom-Favre-Jonsson \cite{Boucksom} is the solution of the NA Calabi conjecture. In this setting, it provides a unique (up to constant) continuous psh potential $\phi_0$ on $(X_K^{an}, c_1(L))$, which solves the \emph{non-archimedean Monge-Amp\`ere (NA MA) equation}
		\[
		\text{MA}( \phi_0) = (L^n)\mu_0,
		\]
		and is called the NA CY potential. 
		It is recently proved \cite{Litheta} that whenever $1\leq \dim Sk(X)\leq n$,  up to suitable normalisation, the K\"ahler potential of $(X_t, \omega_t, \Omega_t)$ converges in the $C^0$-hybrid topology  to the NA CY potential $\phi_0$.

		\item  A choice of SNC model $\pi: \mathcal{X}\to \mathbb{D}$ for the degeneration family $\pi: X\to \mathbb{D}^*$, induces a \emph{retraction map} $r_\mathcal{X}$ from $X_K^{an}$ to the dual complex $\Delta_{\mathcal{X}}$. The essential skeleton $Sk(X)$ can be viewed as a subset of $\Delta_{\mathcal{X}}$. 
		
		\begin{Def}\label{Def:weakcomparison}
	An SNC model satisfies the \emph{weak comparison property} for the NA CY potential $\phi_0$, if there is an open subset $U\subset Sk(X)$ with full Lebesgue measure $\mu_0(U)=1$, such that $\phi_0$ factors through the retraction map $r_{\mathcal{X}}: X_K^{an}\to \Delta_{\mathcal{X}}$ over the open subset $U\subset Sk(X)\subset \Delta_{\mathcal{X}}$. 
		\end{Def}

		\begin{rmk}
			The name `comparison property' comes from a result of C. Vilsmeier \cite{Vilsmeier}, which allows one to compare the NA MA equation with the real Monge-Amp\`ere equation when 
			this property is satisfied. The weak comparison property is preserved under smooth blow ups and blow downs, hence if it holds for one SNC model, then it will hold for all SNC models (\cf Remark \ref{rmk:blowup}). As a small caveat, Def. \ref{Def:weakcomparison} is similar but not identical to the comparison property in \cite{LiNA}, and we will give some exposition in Section \ref{sec:comparisonproperty}.
		\end{rmk}
		
				The arguments in \cite{LiNA} essentially proved the following:
		
		\begin{thm} \label{thm:comparisonpropertySYZ}
			Let $\pi: X\to \mathbb{D}^*$ be a polarised maximal degeneration of  compact CY manifolds. 
			Suppose the weak comparison property is verified for some semistable SNC model, then the metric SYZ conjecture holds for the degeneration family $\pi: X\to \mathbb{D}^*$ . 
		\end{thm}
		
	\end{enumerate}

	The upshot is that the metric SYZ conjecture, which is a priori a PDE problem in complex geometry, is reduced to a question about the NA CY potential, which one hopes to resolve using algebro-geometric inputs.




	The key piece of algebraic geometry turns out to be a \emph{canonical basis} in the graded vector space  $\bigoplus_{l\geq 0} H^0(X_K, lL)$ over the field $K=\C(\!(t)\!)$, satisfying \emph{valuative independence}. Intuitively, such  a basis $\{ \theta_\alpha^l\}$ behaves like the monomials on the affinoid torus, which makes the CY manifold resemble a toric variety, hence should be beneficial for producing the SYZ special Lagrangian $T^n$-fibration.

	\begin{Def}\label{Def:valuativeindependence}
For any $l\geq 0$,  we say a $K$-basis  $\{ \theta_\alpha^l \}  \subset H^0(X_K, lL)$  satisfies \emph{valuative independence}, if for every $x\in Sk(X)$ viewed as the valuation $val_x$, we have 
		\[
		val_x(\sum_\alpha a_\alpha \theta^l_\alpha)= \min_{\alpha: a_\alpha\neq 0}  (val(a_\alpha) + val_x(\theta^l_\alpha) ),\quad \forall a_\alpha\in K.
		\]
		(Here we regard sections of $H^0( X_K, lL )$ as local meromorphic functions using some auxiliary local trivialisation of the line bundle, but the valuative independence condition is independent of the choice of trivialisation.)

	\end{Def}

The search for canonical basis is an active  topic in algebraic geometry.

\begin{enumerate}
	\item  In the Gross-Siebert programme \cite{Gross, Grosstheta1,Grosstheta2, Grosstheta3}, given a polarised maximal degeneration of compact CY manifolds, one can build a graded ring from punctured Gromov-Witten invariants, and define the mirror CY family via the relative Proj construction of this graded ring. Consequently, the mirror family carries a polarisation line bundle, whose section ring has a canonical basis called \emph{theta functions}. This is a vast generalisation for the classical theta functions on abelian varieties.

	\item  It is unknown in general whether the Gross-Siebert theta functions satisfies the valuative independence condition for polarised maximal degeneration of compact CY manifolds. However, in an analogous context for cluster varieties, the Gross-Siebert construction produces a distinguished basis of theta functions for the coordinate ring on certain open Calabi-Yau manifolds, and the very recent work of Cheung-Magee-Mandel-Muller \cite{Travis} verified the valuative independence condition.

	\item  H. Blum and Y. Liu	informed the author that  they can prove the following in full generality (not necessarily assuming maximal degeneration).
	
	\begin{thm}
	\cite{BlumLiu}  Let $(X_K, L)\to \text{Spec}(K)$ be any polarised degeneration of smooth projective CY varieties  over the formal punctured disc, where $L$ is relatively ample. Then there exists some integer $m_0\geq 1$, such that for any $l\geq 0$, there exists a $K$-basis $\{ \theta_\alpha^l \}\subset H^0(X_K, m_0 lL)$ satisfying valuative independence. 
	\end{thm}

	\begin{rmk}
		This relies on deep techniques from birational geometry concerning finite generation of section rings, and construction of multi-parameter test configurations. Blum-Liu's general  result also applies to projective log canonical Calabi-Yau pairs over $\text{Spec}(K)$.
	\end{rmk}

\end{enumerate}

		Let $\pi: X\to \mathbb{D}^*$ be a polarised maximal degeneration of  compact CY manifolds. Up to finite base change, without loss of generality we suppose that there is some semistable model $\mathcal{X}\to \mathbb{D}$. 
	Our main theorem is 
	
	\begin{thm}
	 Assuming for any $l\geq 0$, the $K$-vector space $H^0(X_K, lL)$ admits a basis satisfying valuative independence, then the weak comparison property holds for this degeneration family $\pi: X\to \mathbb{D}^*$.
	\end{thm}

	Together with Thm. \ref{thm:comparisonpropertySYZ} and the exciting upcoming work of Blum-Liu applied to $m_0 L$ instead of $L$, this would \emph{prove the metric SYZ conjecture for all polarised maximal degeneration of compact CY manifolds}. Furthermore, we will interprete the NA MA equation, in terms of an \emph{optimal transport problem} between $Sk(X)$ and an auxiliary probability measure space $(B, \nu)$, which can be constructed via \emph{Okounkov body} techniques.

	\begin{rmk}
	Optimistically, one would hope that $(B,\nu)$ can be identified with the essential skeleton of the mirror family, equipped with the Lebesgue measure. This would lead to an appealing metric version of mirror symmetry, see \cite[Section 6.3]{Litheta} for more discussions.
	\end{rmk}

	\begin{rmk}
	J. Hultgren and M. S. Khalid inform the author that they have an upcoming preprint \cite{HultgrenKhalid}, which defines a cost function explicitly computable from the monodromy of an affine structure. In the case of the Fermat family of cubic curves it agrees with the cost function attained from a valuative independent basis. They also derive a sufficient criterion for when the  associated Kantorovich potential solves a real Monge-Amp\`ere equation, and demonstrate that this is satisfied in some examples. 
	\end{rmk}

	\textbf{Organisation}. We recall some background from non-archimedean geometry in Section \ref{sec:NA}, including the Boucksom-Favre-Jonsson solution to the NA Monge-Amp\`ere equation, the relative volume interpretation for the NA Monge-Amp\`ere energy due to Boucksom-Eriksson, and some exposition on how the weak comparison property implies the metric SYZ conjecture.

	The new contents are contained in Section \ref{sec:valuativeindependence}, \ref{sec:NAMAoptimaltransport}.
	In Section \ref{sec:valuativeindependence} we assume the existence of the valuative independent basis  $\{ \theta_\alpha^l \}$ for $\bigoplus_{l\geq 0} H^0(X_K, lL)$, and derive consequences on the NA semipositive potentials. We will extract a cost function $c(x,p): Sk(X)\times B\to \R$ by taking suitable asymptotic limits for $\log |\theta_\alpha^l|$, which allows us to introduce a distinguished class of functions $\mathcal{P}_c\subset C^0(Sk(X))$ (Section \ref{sec:costfunction}). The functions in $\mathcal{P}_c$ can be canonically identified with continuous semipositive potentials on $(X,L)$ satisfying a domination property (Section \ref{sec:Pcvspsh}). We use the tool of Okounkov bodies to extract more refined information for the cost function (Section \ref{sec:gradientconvexhullOkounkovbody}, \ref{sec:Okounkovcost}), leading to a criterion for the semipositive potential to factor through the retraction map $r_\mathcal{X}: X_K^{an}\to \Delta_\mathcal{X}$ over local open neighbourhoods inside $Sk(X)$ (Section \ref{sec:factorisation}), and an alternative formula for the NA Monge-Amp\`ere energy (Section \ref{sec:relativevolume}).

	In Section \ref{sec:NAMAoptimaltransport} we draw consequences  for a semipositive potential $\varphi$ solving the NA MA equation
$
	\text{MA}(\varphi)= (L^n) \mu,
$
	assuming that $\mu$ is a probability measure supported on $Sk(X)$, and is absolutely continuous with respect to the Lebesgue measure on $Sk(X)$.  The special case  $\mu=\mu_0$ corresponds to the NA CY potential. In this generality, we will prove that the weak comparison property holds for $\varphi$, and the potential $\varphi$ admits an optimal transport interpretation. A crucial ingredient is that certain bad sets have null measure, and the proof involves a delicate application of the orthogonality property for psh envelops.

	\begin{Acknowledgement}
	The author is sponsored by the Royal Society URF. he would like to especially thank H. Blum and Y. Liu for explaining their upcoming work on the existence of valuative independent basis. He also thanks S. Boucksom and M. Gross for past discussions.
	\end{Acknowledgement}

	\section{Non-archimedean pluripotential theory}\label{sec:NA}

		We recall some notions from non-archimedean geometry, based on the work of Boucksom et al. \cite{Boucksom, Boucksom1, Boucksomnew1, Boucksomsemipositive, Boucksomsurvey}; see also \cite[section 5]{Boucksom1}\cite[section 2]{Liintermediate}.

	\subsection{Models and dual complex}

	We denote $R=\C[\![ t ]\!]$, and $K=\C(\!(t)\!)$ with the standard discrete valuation.

	Given a polarised meromorphic degeneration family of $n$-dimensional smooth projective varieties $\pi: X\to \mathbb{D}^*$, upon base change to the formal punctured disc $\text{Spec}(K)$, we obtain a degeneration family $\pi: X_K\to \text{Spec}(K)$.  Let $X_K^{an}$ denote the  \emph{Berkovich space}. Given an affine variety over $K$, the Berkovich space is the space of semivaluations on the ring of functions extending the standard discrete valuation on $K$, and in general $X_K^{an}$ is obtained by gluing the affine pieces.	The semivaluations $v$ are equivalently thought of as multiplicative seminorms $|\cdot |=e^{-v}$, satisfying the ultrametric property
	\[
	|f+g|\leq \max\{ |f|+|g|  \}, \quad |fg|=|f||g|.
	\]

	A \emph{model} of $X_K\to \text{Spec}(K)$ (resp. $X\to \mathbb{D}^*$) is a normal projective scheme $\mathcal{X}$, flat and of finite type over $\text{Spec}(R)$ (resp. the complex disc $\mathbb{D}$), together with an identification with $X_K\to \text{Spec}(K)$ (resp. $X\to \mathbb{D}^*$) away from the special fibre. For any model $\mathcal{X}$ and every irreducible component $E$ of $\mathcal{X}_0$, we write the central fibre as $\mathcal{X}_0=\sum b_iE_i$, and $E_J= \cap_{i\in J}E_i$. We say $\mathcal{X}$ is SNC if $\mathcal{X}$ is regular and $\mathcal{X}_0$ has SNC support, and an SNC model is semistable if all the divisor components on $\mathcal{X}_0$ have multiplicity one. By Hironaka resolution and semistable reduction, up to finite base change, we shall without loss assume that there exists some semistable SNC model. A model of the line bundle $L\to X_K$ is a $\Q$-line bundle $\mathcal{L}$ on a proper model $\mathcal{X}$, together with an identification $\mathcal{L}|_{X_K}=L$.

	The \emph{dual complex} $\Delta_{\mathcal{X}}$ of an SNC model $\mathcal{X}$ is the simplicial complex whose vertices correspond to the irreducible components $E_i$ of the central fibre, and whose faces correspond to the connected components of nonempty intersection strata $E_J$, and are identified as the simplex 
	$\Delta_J\simeq \{  x\in \R_{\geq 0}^{|J|}| \sum_{i\in J} b_ix_i=1 \}.
	$
	Given an SNC model over $\text{Spec}(R)$, there are two natural comparison maps between its dual complex and the Berkovich space.
	
	\begin{enumerate}
		\item   There is a continuous \emph{retraction map}  $r_{ \mathcal{X} }: X_K^{an}\to \Delta_{\mathcal{X} } $. 
		Any point $e^{-v}\in X_K^{an}$ admits
		a centre on $\mathcal{X}$. 
		This is the unique scheme theoretic point
		$\xi \in X_0$ such that 
		$|f|_x \leq 1$ for $f\in \mathcal{O}_{\mathcal{X},\xi}$
		and $|f|_x < 1$ for
		$f\in m_{\mathcal{X},\xi}$.
		Let $J \subset I$ 
		be the maximal subset such that
		$\xi\in E_J$. Then
		$r_{ \mathcal{X} }(x) \in \Delta_{\mathcal{X} }$ 
		is determined by evaluating
		$-\log |z_j|_x$ for $j \in J$.
		These retraction maps form commutative diagrams under further blow ups of SNC models.
	 The tower of retraction maps defines a homeomorphism
		\[
		X_K^{an}\simeq  \varprojlim_{\text{snc models} } \Delta_{ \mathcal{X} }.
		\]


		\item   There is an \emph{embedding map} $emb_{\mathcal{X}}: \Delta_\mathcal{X}\to X_K^{an}$, which is a homeomorphism onto its image. Given $x\in \Delta_J$, we need to associate a quasi-monomial valuation $val_x$, regarded as a point in $X_K^{an}$. Given any local function $f\in \mathcal{O}_{\mathcal{X},\eta}$ where $\eta$ is the generic point of $E_J$, we can Taylor expand
		\[
		f=\sum_{\alpha\in \N^{|J|}} f_\alpha z_0^{\alpha_0}\ldots z_{|J|}^{\alpha_{|J|}},
		\] 
		where $z_i$ are local equations of the divisors $E_i$, and $f_\alpha$ is either a unit or zero. Then $val_x$ is defined as the minimal weighted vanishing order:
		\[
		val_x(f)= \min\{   \langle x, \alpha\rangle  |  f_\alpha\neq 0  \}.
		\]
		In particular, the vertices of $\Delta_J$ map to the divisorial valuations. The requirement $\sum b_ix_i=1$ comes from $val(t)=1$, where $t\in R$ is the uniformizer. The composition $r_\mathcal{X}\circ emb_\mathcal{X}=id_{\Delta_\mathcal{X}}$.	Henceforth we identify the dual complex with its image in $X_K^{an}$.

	\end{enumerate}

	Moreover, there is a natural topology on $X_K^{an}\cup X$, known as the \emph{hybrid topology}, which describes $X_K^{an}$ as a limiting object of the complex manifolds $X_t$ \cite[Appendix]{Boucksom1}\cite[Section 1.2]{PilleSchneider}; a one-page recap is in \cite[section 2.2]{Liintermediate}.

	\subsection{Essential skeleton and volume form asymptote}\label{section:essentialskeleton}

	Suppose $X$ is a meromorphic degeneration family of smooth projective CY manifolds $X_t$ over some punctured disc $\mathbb{D}^*$, and $X$ is equipped with a holomorphic volume form $\Omega$, which induces the holomorphic volume forms $\Omega_t$ on the $X_t$ via $\Omega= \Omega_t\wedge dt$. The \emph{essential skeleton} is the subset of $X_K^{an}$ where the log discrepancy function takes the minimal value $\kappa_0$ \cite[Section 3.1.1]{NicaiseXu}.  Up to passing to semistable reduction and multiplying $\Omega$ by a suitable power of $t$, we may assume $\kappa_0=0$.

	We follow \cite{Boucksom1} to sketch the computation of the CY volume form asymptote. Let $\mathcal{X}$ be an SNC model of the degeneration family, whose dual complex is $\Delta_{\mathcal{X}}$. We write the central fibre as $\mathcal{X}_0=\sum b_iE_i$, and the relative log canonical divisor as $\sum a_i E_i$ (so the log discrepancy function takes the value $\frac{a_i}{b_i}$ at the divisorial point corresponding to $E_i$). We consider an intersection stratum $E_J= \cap_{i\in J} E_i$, and let $z_0,\ldots z_n$ be local coordinates on $E_J$ bounded away from the deeper intersection strata, such that $z_0,\ldots z_k$ are the local defining equation for the $E_i$ with $i\in J$. Locally there is a nowhere vanishing holomorphic function $u_J$ such that
	\[
	\Omega=u_J z_0^{a_0+b_0-1}\ldots z_k^{a_k+b_k-1} dz_0\wedge \ldots  dz_n,
	\]
	and we can arrange
	$
	t= z_0^{b_0}\ldots z_k^{b_k} .
	$
	Then
	\begin{equation}\label{holovollocal}
		\Omega_t= b_0^{-1} u_J z_0^{a_0}\ldots z_k^{a_k} d\log z_1\wedge \ldots d\log z_k\wedge dz_{k+1}\wedge \ldots dz_n.
	\end{equation}
	By the non-negativity of the log discrepancy function, we have $a_i\geq 0$ for $i=0,\ldots k$.

	For $0<|t|\ll 1$, we have a logarithm map $\text{Log}_\mathcal{X}: X_t\to \Delta_\mathcal{X}$, which can be chosen locally as
	\[
	\text{Log}_{\mathcal{X}} (z)=(x_0,\ldots x_k):=  \frac{1}{\log |t|}(   \log |z_0|, \ldots \log |z_k|      ),
	\]
	so that $\sum_0^k b_i x_i=1$ by $t= z_0^{b_0}\ldots z_k^{b_k}$. Passing to polar coordinates allows one to compute the volume asymptote.
	If at least one exponent $a_i>0$, then the local contribution to the volume integral $\int_{X_t} \Omega_t\wedge \overline{\Omega}_t$ is suppressed by a factor $O(\frac{1}{|\log |t||})$. The divisors with $a_i=0$ correspond to the vertices of a subcomplex of $\Delta_{\mathcal{X}}$, which is naturally identified with the essential skeleton $Sk(X)\subset X_K^{an}$ under the embedding map.  The dimension $m=\dim Sk(X)$ then equals the largest value of $k$ such that $a_0=\ldots =a_k=0$, so the total volume integral $\int_{X_t} \Omega_t\wedge \overline{\Omega}_t=O(|\log |t||^m)$. We denote the normalised CY measure as
	\[
	\mu_t= \frac{  \Omega_t\wedge \overline{\Omega}_t }{   \int_{X_t} \Omega_t\wedge \overline{\Omega}_t      } ,
	\]
	which is a probability measure on $X_t$.
	Boucksom-Jonsson \cite{Boucksom} show that

	\begin{thm}\label{Volumeconvergence}
		Under the hybrid topology, the measures $\mu_t$ on $X_t$ converge weakly to a probability measure $\mu_0$ supported on $Sk(X)\subset X_K^{an}$, which puts no measure on the lower dimensional faces of $Sk(X)$, and agrees with a suitably normalized Lebesgue measure $\mu_0$ on the $m$-dimensional faces of $Sk(X)$.
		More concretely, for any continuous function $f$ on the hybrid space $X_K^{an}\cup X$, we have
		\[
		\int_{X_t} fd\mu_t\to \int_{X_K^{an}} fd\mu_0.
		\]
	\end{thm}

In the case of \emph{maximal degeneration}, $\dim_\R Sk(\mathcal{X})=n$, and the measure convergence can be understood explicitly in a semistable SNC model. Near one of the deepest SNC intersection strata $E_J$ corresponding to an $n$-dimensional simplex in $Sk(\mathcal{X})$, locally the CY volume form is proportional to
	\begin{equation}\label{volumeasymptote}
		\sqrt{-1}^{n^2} \Omega_t\wedge \overline{\Omega}_t =  |u_J|^2  \prod_1^n \sqrt{-1} d\log z_i \wedge d\log \bar{z}_i.
	\end{equation}
	Here $u_J$ limits to its value $u_J(E_J)$ at the point stratum $E_J$, which is called the Poincar\'e residue of $\Omega$, and is easily seen to be independent of the choice of coordinates $z_i$. It is a consequence of the residue theorem on Riemann surfaces that $|u_J(E_J)|^2$ is independent of such $J$ \cite[Thm. 7.1]{Boucksom1}. Thus the pushforward to $\Delta_{\mathcal{X}}$ of the normalised CY measure $\mu_t$  converges smoothly in the interior of $\Delta_J$ to a constant multiple of the \emph{Lebesgue measure}:
	\begin{equation}\label{measureconvergence}
		\text{Log}_{ \mathcal{X} *  }\mu_t= \text{Log}_{ \mathcal{X}*  }\frac{ \Omega_t\wedge \overline{\Omega}_t }{  \int_{X_t}  \Omega_t\wedge \overline{\Omega}_t    } \xrightarrow{t\to 0}    \mu_0:=C_0  dx_1\ldots dx_n.
	\end{equation}
	Notice $dx_1\ldots dx_n$ is canonically defined due to the presence of an integral affine structure on $\Delta_J$. The constant in (\ref{measureconvergence}) is independent of $J$ and its sole purpose is to make $\mu_0$ a probability measure.

	\subsection{Fubini-Study and semipositive metrics}

	We now discuss line bundles and metrics on $X_K^{an}$.
	A GAGA principle says the line bundles on $X_K^{an}$ correspond to the line bundles $L$ on the scheme $X_K$. A \emph{continuous metric} on $L$ assigns to each local section $s$ a nonnegative continuous  local function $\norm{s}$ on open subsets of $X_K^{an}$, compatible with the sheaf structure, such that $\norm{fs}(x)= |f|_x\norm{s}(x)$, and $\norm{s}>0$ if $s$ is a local frame of $L$.
	Given a continuous metric, any other continuous metric on $L$ is of the form $\norm{\cdot} e^{-\phi}$ for some $\phi\in C^0(X^{an})$. As such $\phi$ is referred to as a \emph{potential} function.

For any model $\Q$-line bundle $\mathcal{L}\to \mathcal{X}$ of $L\to X$, we can associate a unique metric $\norm{\cdot}_{ \mathcal{L} }$ on $L$ with the following property: if $s$ is a nowhere vanishing local section of the line bundle $\mathcal{L}^{\otimes k}$ for some $k\in \N$, on an open set $\mathcal{U}\subset \mathcal{X}$, then $\norm{s}_{ \mathcal{L} } \equiv 1$ on $\mathcal{U}\cap X_K$.  Such metrics are called \emph{model metrics}, which are dense within the continuous metrics on $L\to X_K^{an}$. In the special case of trivial $L=\mathcal{O}\to X$, then $\mathcal{L}\to \mathcal{X}$ corresponds to vertical $\Q$-divisors, which 
gives the \emph{model functions}. These functions form a dense subset of  $C^0(X_K^{an})$.

	A norm on a finite dimensional $K$-vector space $V$ is called \emph{ultrametric} if $\norm{x+y}_V\leq \max(\norm{x}_V, \norm{y}_V)$. Concretely, one can select an orthogonal $K$-basis $s_1,\ldots s_N$, namely
		\[
	\norm{ a_1s_1+ \ldots+ a_N s_N}_V=\max\{ |a_1|\norm{s_1}_V,\ldots ,|a_N|  \norm{s_N}_V  \}, \quad \forall a_i\in K.
	\]
	Now let $L$ be a relatively ample line bundle, and let $l\geq 1$. Given an ultrametric norm on the $K$-vector space $V=H^0(X_K,lL)$ which generates the line bundle $L^{\otimes l}\to X$, the \emph{NA Fubini-Study metric} on the line bundle $L\to X_K^{an}$ is defined by
	\[
	\norm{s}_{FS,l}(x)=  \inf_{ \tilde{s}\in V, \tilde{s}(x)=s^{\otimes l}(x)} \norm{ \tilde{s}}_V^{1/l}, \quad \forall x\in X_K^{an}.
	\]
	Concretely, given an orthogonal $K$-basis $s_1, \ldots s_N$ for $V$, we can write
  $\norm{s_j}_V= e^{c_j} $ for $c_j\in \R$, and  then
	\begin{equation*}
		\norm{s}_{FS ,l}(x)= \frac{ |s(x)|}{ \max_j \{  |s_j(x)|e^{-c_j}  \}^{1/l}   }, \quad \forall x\in X^{an}.
	\end{equation*}
	This can be recast in terms of potentials
	\begin{equation}\label{NAFubiniStudy}
		\phi_{FS,l}= \frac{1}{l} \max_{1\leq i\leq N}(\log |s_i| -c_i),
	\end{equation}
	where to evaluate $|s_i|$, we implicitly use a local trivializing section of $L$ coming from the choice of a model $\Q$-line bundle.

	A \emph{continuous semipositive metric} $\norm{\cdot}=\norm{\cdot}_{\mathcal{L}}e^{-\phi}$ on $L$ can be then defined as a uniform limit of some sequence of NA Fubini-Study metrics. The space of these psh potentials $\phi$ is denoted as $\text{CPSH}(X_K^{an}, \mathcal{L})$. These potentials satisfy an \emph{equicontinuity} estimate.

	\begin{prop}\cite[Chapter 6]{Boucksomsemipositive}
	Let $L\to X$ be a relatively ample line bundle, and let $\mathcal{L}\to \mathcal{X}$ be a fixed model $\Q$-line bundle. Then there exists a constant depending only on $(\mathcal{X},\mathcal{L})$, such that for every $\phi\in \text{CPSH}(X_K^{an}, \mathcal{L})$, the composition $\phi\circ emb_{\mathcal{X}}$ is convex and $C$-Lipschitz on each face of $\Delta_\mathcal{X}$.
	\end{prop}

This equicontinuity result has the following useful consequence.

	\begin{lem}\cite[Lemma 2.3]{Litheta}\label{lem:Lip}
		(Uniform Lipschitz estimate)
		Let $l\geq 1$, and let $L\to X$ be a relatively ample line bundle. Given a model $\mathcal{L}\to \mathcal{X}$, the valuation $x\mapsto val_x(s)$ for any nonzero $s\in H^0(X_K^{an}, lL)$ makes sense as a function on the dual complex $\Delta_{\mathcal{X}}$ using the local trivialisation of $\mathcal{L}$. On each face of $\Delta_{\mathcal{X}}$, we have a uniform estimate
		\[
		l^{-1}|val_x(s)- val_{x'}(s) | \leq C|x-x'|, \quad \forall x,x'\in \Delta_J,
		\]
		where the constant is independent of $x,x',s,l$.
	\end{lem}

	\subsection{NA Monge-Amp\`ere equation}\label{sec:NAMAeqn}

	The non-archimedean Monge-Amp\`ere (NA MA) measure can be defined through \emph{intersection theory}. 
	Given any model $\Q$-line bundle $\mathcal{L}\to \mathcal{X}$ for  $L\to X_K$, we write $\mathcal{X}_0= \sum_i b_i E_i$, then the NA MA measure for the model metric $\norm{\cdot}_ {\mathcal{L} }$ is the signed atomic measure
	\[
	\text{MA}( \norm{\cdot}_{ \mathcal{L} } )= \sum_{E_i} b_i (\mathcal{L}^n\cdot E_i) \delta_{E_i},
	\] 
	where $\delta_{E_i}$ denotes the delta measure supported at the divisorial point  associated to $E_i$. If $\norm{\cdot}_{\mathcal{L}  }$ is furthermore semipositive, then the intersection numbers are non-negative, so $\text{MA}( \norm{\cdot}_{ \mathcal{L} } )$ is a measure, and the total measure is the intersection number $(L^n)$.

	Fixing a reference model $\Q$-line bundle $\mathcal{L}\to \mathcal{X}$, then any $\phi\in \text{CPSH}(X_K^{an}, \mathcal{L})$ can be $C^0$-approximated by the potentials of a sequence of  semipositive model metrics \cite[Cor. 8.8]{Boucksomsemipositive}, and its NA MA measure $\text{MA}(\phi)$ can be defined as the unique weak limit Radon measure of the NA MA measures for its approximation sequence \cite[Cor. 3.5]{Boucksom}.

	The central result of NA pluripotential theory is the solution of the NA Calabi conjecture.
	
	\begin{thm}\label{thm:BFJ}
		\cite{Boucksomsurvey, Boucksom} Let $L\to X$ be a relatively ample line bundle, and let $\mathcal{L}\to \mathcal{X}$ be a fixed model $\Q$-line model. Let $\mu$ be any Radon probability measure supported on some fixed dual complex inside $X_K^{an}$. Then there is a unique (up to an additive constant) $\phi \in \text{CPSH}(X_K^{an}, \mathcal{L})$ solving the NA MA equation $\text{MA}(\phi )=(L^n)\mu$.
	\end{thm}

	In particular, when $(X,L)\to \mathbb{D}^*$ is a polarised meromorphic degeneration of compact CY manifolds, we can take $\mu_0$ to be the normalised Lebesgue measure on $Sk(X)\subset X_K^{an}$. The unique (up to constant) solution $\phi_0\in \text{CPSH}(X_K^{an}, \mathcal{L})$ to the NA MA equation
	\begin{equation}
	\text{MA}(\phi_0)= (L^n) \mu_0,
	\end{equation}
	is called the \emph{NA CY potential}, and $\norm{\cdot}_{\mathcal{L}} e^{-\phi_0}$ is called the NA CY metric.
	
	The original proof by Boucksom-Favre-Jonsson uses a variational strategy. They introduce a concave functional $E: \text{CPSH}(X_K^{an}, \mathcal{L})\to \R$, called the \emph{Monge-Amp\`ere energy}, whose first variation is formally the NA MA measure. Then they consider the convex functional $F_{\mu}: \text{CPSH}(X_K^{an}, \mathcal{L})\to \R$,
	\begin{equation}\label{eqn:Fmu01}
		F_{\mu}(\phi)= -  \frac{1}{(L^n)} E(\phi)+ \int_{X_K^{an}} \phi d\mu =- \frac{1}{(L^n)}E(\phi)+ \int_{Sk(X)} \phi d\mu .
	\end{equation}
Upon enlarging the class of semipositive metrics, one can use compactness to extract a minimiser $\phi$, and prove that $\phi$ in fact solves the NA MA equation by a subtle \emph{differentiability property} for the Monge-Amp\`ere energy, and furthermore $\phi$ has $C^0$-regularity via capacity estimates \cite{Boucksom, Boucksomsurvey}.

	\begin{rmk}
	In our context of polarised degeneration for compact Calabi-Yau manifolds, one can show that up to suitable normalisation, the CY potentials on $(X_t, \frac{1}{|\log |t||} c_1(L))$ converge in the $C^0$-hybrid topology to the NA CY potential $\phi_0$ on $X_K^{an}$. 
	\end{rmk}
	
The NA MA equation satisfies the following \emph{domination principle}.

\begin{lem}\label{lem:dominationproperty}
\cite[Lemma 8.4]{Boucksom} Let $\varphi, \psi \in CPSH(X_K^{an}, \mathcal{L})$, suppose $\mu= \text{MA}(\psi)$ is supported in the dual complex $\Delta_{\mathcal{X}}$ of some SNC model $\mathcal{X}$, and $\varphi\leq \psi$ $\mu$-a.e. Then $\varphi\leq \psi$ on $X_K^{an}$.
\end{lem}

	\subsection{Relative volume interpretation}\label{sect:relativevolume}

	Boucksom-Eriksson \cite{Boucksomnew1}  interpretated $E(\phi)$ in terms of relative volume as follows. Given two ultrametric norms $\norm{}, \norm{}'$ on the $K$-vector space $ H^0(X_K, lL)  $, the relative volume is
	\[
	vol(  \norm{}, \norm{}'  )=  \log \frac{  \det(\norm{}')  } {   \det(\norm{})   }.
	\]
	Concretely, if both $\norm{}$ and $\norm{}'$ are diagonal with respect to some NA orthogonal basis $s_1,\ldots s_N$ over the $K$-vector space $V$, namely
	\[
	\norm{ \sum a_i s_i}= \max_i |a_i| \norm{s_i}, \quad \norm{ \sum a_i s_i}'= \max_i |a_i| \norm{s_i}',\quad \forall a_i\in K,
	\]
	and $\norm{s_i}'= \lambda_i \norm{s_i}$, then $vol( \norm{}, \norm{}'    )= \log  \prod_i \lambda_i$.

	Given a continuous psh potentials $\phi, \psi \in \text{CPSH}(X_K^{an}, \mathcal{L})$, we can define a sequence of NA norms on the $K$-vector spaces $H^0(X_K, lL)$ for $l$ large enough,
	\[
	\norm{s}_{l\phi}= \sup_{X_K^{an}} |s|_{\mathcal{L}} e^{-l\phi},
	\]
	and similarly we define $\norm{s}_{l\psi}$. 
	By \cite[Theorem A]{Boucksomnew1}, the \emph{relative MA energy} admits the formula
	\[
	E(\phi)-E(\psi)= E(\phi, \psi)=  \lim_{l\to +\infty} \frac{n!}{ l^{n+1}} vol(   \norm{\cdot}_{l\phi}, \norm{\cdot}_{l\psi}     ) .
	\] 
	This determines $E(\phi)$ up to the choice of an additive constant.

	\subsection{Envelop and orthogonality property}\label{sec:envelop}

	Let $L\to X$ be a relatively ample line bundle, and 
	let $\mathcal{L}\to \mathcal{X}$ be a fixed model $\Q$-line bundle.
	For any continuous function $f: X_K^{an}\to \R$, we define the \emph{psh envelop} as 
	\[
	P(f):= \sup \{    \varphi:  \varphi\in \text{CPSH}(X_K^{an}), \mathcal{L}), \varphi\leq f        \}.
	\]
	Then $P(f)\in \text{CPSH}(X_K^{an}, \mathcal{L})$ and is the largest element of $ \text{CPSH}(X_K^{an}, \mathcal{L})$  dominated by $f$ \cite[Prop. 8.2]{Boucksomsemipositive}.

	We now recall the \emph{orthogonality property} for psh envelops, whose proof is based on the orthogonality of the asymptotic Zariski decomposition. It underpins the \emph{differentiability property for Monge-Amp\`ere energy}, hence is a crucial ingredient in the theory for NA MA equation.

	\begin{prop}\cite[Appendix A]{Boucksom}
For every  $f\in C^0(X_K^{an})$, we have
	\[
	\int_{X_K^{an}} (f- P(f)) \text{MA}( P(f)  )=0.
	\]
	\end{prop}

	We will use the orthogonality property through the following consequence.
	
	\begin{cor}\label{cor:orthogonality}
	\cite[proof of Cor. 7.3]{Boucksom}  Let $\phi\in \text{CPSH}(X_K^{an}, \mathcal{L})$, and denote $\mu= \text{MA}(\phi)$. Let $f$ be any given model function on $X_K^{an}$. Then for small $\tau$, 
	\[
	\mu( \{      P(\phi+\tau f) <  \phi+\tau  f        \}   ) \leq C |\tau|,
	\]
for some constant $C$ independent of $\tau$. 
	\end{cor}

\subsection{Weak comparison property, metric SYZ conjecture}\label{sec:comparisonproperty}

Suppose $(X,L)\to \mathbb{D}^*$ is a polarised \emph{maximal degeneration} of compact CY manifolds. We fix a reference model $\Q$-line bundle $\mathcal{L}$, and let $\phi_0\in  \text{CPSH}(X_K^{an}, \mathcal{L})$ be the NA CY potential.
We recall the weak comparison property in Def. \ref{Def:weakcomparison}.

		\begin{Def}
	An SNC model satisfies the \emph{weak comparison property} for the NA CY potential $\phi_0$, if there is an open subset $U\subset Sk(X)\subset  \Delta_{\mathcal{X}}$ with full Lebesgue measure $\mu_0(U)=1$, such that $\phi_0 =\phi_0\circ r_{\mathcal{X}} $ on $r_\mathcal{X}^{-1}(U)$ under the retraction map $r_{\mathcal{X}}: X_K^{an}\to \Delta_{\mathcal{X}}$.
\end{Def}

\begin{rmk}\label{rmk:blowup}
	The weak comparison property is preserved under smooth blow ups and blow downs for SNC models. If the blow up centre is not one of the intersection strata $E_J=\cap_{i\in J} E_i$ for divisor components $E_i$, then the essential skeleton $Sk(X)$ is unaffected, and the retraction map $r_{\mathcal{X}}: X_K^{an}\to \Delta_{\mathcal{X}}$ is unaffected over the open $n$-dimensional faces of $Sk(X)\subset \Delta_{\mathcal{X}}$. 
Otherwise $\Delta_{\mathcal{X}}$ is subdivided as a simplicial complex, but the retraction map $r_{\mathcal{X}}: X_K^{an}\to \Delta_{\mathcal{X}}$ remains the same. Thus each blow up and blow down can only affect the retraction map over a closed subset of $Sk(X)$ with Lebesgue measure zero.
\end{rmk}

\begin{thm}\cite{LiNA}
	Suppose the weak comparison property is verified for some semistable SNC model $\mathcal{X}$, then the metric SYZ conjecture holds for the degeneration family $\pi: X\to \mathbb{D}^*$ . 
\end{thm}

We now give a sketchy exposition for the proof strategy, taking some shortcut via the $C^0$-convergence for CY potential \cite[Thm. 1.2]{Litheta}.

\begin{enumerate}
	\item  Without loss the open set $U$ is contained in the union of the open $n$-dimensional faces $\Delta_J^0\subset Sk(X)$.   By a result of C. Vilsmeier \cite{Vilsmeier}, under the weak comparison property, then on $\Delta_J^0\cap U$, the pushforward of the NA MA measure equals the real Monge-Amp\`ere measure of the convex function $\phi|_{\Delta_J^0}$ up to a factor $n!$,
	\[
 \text{MA}_\R (\phi_0 |_{\Delta_J^0 \cap U} )  =\frac{1}{n!}	r_{\mathcal{X}* } \text{MA}( \phi_0 )= \frac{(L^n)}{n!}\mu_0  = \frac{(L^n)}{n!} C_0 dx_1\ldots dx_n.
	\]
	By the regularity theory for the real Monge-Amp\`ere equation due to C. Mooney, the solution $\phi_0$ is smooth away from a closed set of Hausdorff dimension at most $n-1$ inside $U\cap \Delta_J^0$. We can delete this singular set from $U$, so that 
	 $\phi_0$ is a smooth solution to the \emph{real Monge-Amp\`ere equation} 
	\begin{equation}
   \det(D^2 \phi_0)=  \frac{C_0}{n!} (L^n).
	\end{equation}

	\item  For any small $\epsilon>0$, we let 
$
	U_\epsilon =\{     x\in U: \text{dist}( x, \partial U )>\epsilon      \}, 
$
so $U_\epsilon$ is properly contained in $U$. 
	By fixing $\epsilon$ small enough depending on $\delta$, we can ensure that $\mu_0(U_\epsilon) > 1-\delta/2$.

	As in Section \ref{section:essentialskeleton}, near the depth-$(n+1)$ intersection strata $E_J$, we have local holomorphic coordinates $z_0\ldots z_n$ subject to $t=z_0\ldots z_n$, and we write 
	\[
		\text{Log}_{\mathcal{X}} (z)=(x_0,\ldots x_k):=  \frac{1}{\log |t|}(   \log |z_0|, \ldots \log |z_k|      ),
	\]
Using the measure convergence caluclation in Section \ref{section:essentialskeleton}, the set $\text{Log}_\mathcal{X}^{-1} (U_\epsilon)\subset X_t$ has $\mu_t$-measure at least $1-\delta$.

	\item 
Applying the $C^0$-hybrid convergence result in \cite{Litheta}, and using that $\phi=\phi\circ r_\mathcal{X}$ over $U\subset Sk(X)$, we can write the CY metrics $(X_t, \omega_{CY,t})$ on the subset $\text{Log}_\mathcal{X}^{-1} (U_{\epsilon/3})\subset X_t$  as $\omega_{CY,t}= dd^c \phi_{CY,t}$ for some suitable local potential $\phi_{CY,t}$, such that for the fixed small $\epsilon>0$,
\[
\norm{ \phi_{CY,t}-  \phi_0\circ \text{Log}_\mathcal{X} }_{C^0(   \text{Log}_\mathcal{X}^{-1} (U_{\epsilon/3})  )} \to 0, \quad t\to 0.
\]
The CY metric condition can be written as the PDE
\[
(dd^c \phi_{CY,t} )^n= \omega_{CY,t}^n= \frac{ (L^n) }{ |\log |t||^n   } \mu_t.
\]

	\item  Using the real Monge-Amp\`ere equation on $\phi_0$, we can pass to a local universal cover of $\text{Log}_\mathcal{X}^{-1} (U_{\epsilon/3})$, and apply the Savin small perturbation theorem as in \cite[Thm. 4.8]{LiNA}, to improve the $C^0$-convergence to $C^\infty$-convergence on the shrunken region $\text{Log}_\mathcal{X}^{-1} (U_{\epsilon/2})$,
	\[
	\norm{ \phi_{CY,t}-  \phi_0\circ \text{Log}_\mathcal{X} }_{C^{k+2}(   \text{Log}_\mathcal{X}^{-1} (U_{\epsilon/2})  )} \to 0, \quad t\to 0.
	\]
	for an arbitrarily big $k\in \N$. Thus we obtain the $C^\infty$ metric convergence
	\[
	\norm{ \omega_{CY,t}-    dd^c (\phi_0\circ \text{Log}_\mathcal{X}) }_{C^k(   \text{Log}_\mathcal{X}^{-1} (U_{\epsilon/2})  )} \to 0, \quad t\to 0.
	\]

	\item In the semiflat model metric $  dd^c (\phi_0\circ \text{Log}_\mathcal{X})$, the map $\text{Log}_\mathcal{X}$ is a special Lagrangian $T^n$-fibration. When $t$ is sufficiently small depending on $\delta$, 
	by a result of Y. G. Zhang based on the implicit function theorem, we can make a $C^\infty$-small perturbation to obtain a special Lagrangian $T^n$-fibration, on a subset of $\text{Log}_\mathcal{X}^{-1} (U_{\epsilon/2}) \subset X_t $ that contains $\text{Log}_\mathcal{X}^{-1} (U_{\epsilon})  $ (See \cite[Thm. 4.9]{LiNA}).

	The upshot is that there is a special Lagrangian $T^n$-fibration on some open subset of $X_t$, which has $\mu_t$-measure at least $1-\delta$. The metric SYZ conjecture follows.

\end{enumerate}

	\section{Valuative independence and potential theory}\label{sec:valuativeindependence}

From henceforth $(X,L)\to \mathbb{D}^*$ is a polarised maximal degeneration of CY manifolds. Without loss we can take a semistable model $\mathcal{X}$, together with a reference model line bundle $\mathcal{L}\to \mathcal{X}$, so that  continuous semipositive metrics $\norm{\cdot}= \norm{\cdot}_{\mathcal{L}} e^{-\phi}$ on $L\to X_K^{an}$ can be viewed as potentials $\phi\in \text{CPSH}(X_K^{an}, \mathcal{L})$.

In this Section, we will assume the existence of  the valuative independent basis, and derive consequences on the class of NA semipositive potentials.

\begin{rmk}
	Many ideas here have been foreshadowed in \cite[Section 5]{Litheta}. Whereras in \cite[Section 5]{Litheta} we are given a parameter space $B$ for the cost function $c(x,p): Sk(X)\times B\to \R$, carrying an a priori polyhedral structure, here the analogue for $B$ has to be built from scratch.
\end{rmk}

	\subsection{Valuative independence}\label{sect:valuativeindependence2}

	Let $l\geq 0$, and let $\{  \theta_\alpha^l \}\subset H^0(X_K, lL)$ be a basis over the discretely valued field $K=\C(\!(t)\!)$. Using some local trivialisation $\tau$ provided by $\mathcal{L}$, we can regard $\theta_\alpha^l$ as local meromorphic functions $\theta_\alpha/\tau$. We view the points on $x\in Sk(X)\subset X_K^{an}$ as valuations. The valuative independence condition in Def. \ref{Def:valuativeindependence} reads
	\[
		val_x(\sum_\alpha a_\alpha \theta^l_\alpha/\tau)= \min_{\alpha: a_\alpha\neq 0}  (val(a_\alpha) + val_x(\theta^l_\alpha/\tau) ),\quad \forall a_\alpha\in K.
\]
Notice that the valuative independence condition is unaffected if we change the trivialising section $\tau$ to $\tau'$, because $val_x(\tau'/\tau)$ would cancel out on both sides. Thus we will omit $\tau$ from the notation.
Viewing valuations as norms, this is equivalent to
	\begin{equation}
	| \sum_\alpha a_\alpha \theta_\alpha^l  | (x)=  \max_\alpha |a_\alpha |  |\theta_\alpha^l| (x).
	\end{equation}

We now explain the characterisation for valuative independence in \cite[Section 3.2]{Litheta}, in the special case of maximal degenerations $\dim Sk(X)=n$. Let $E_J=\cap_{i\in J} E_i$ be a depth $n+1$ intersection of SNC divisor components $E_i$ on the central fibre, such that $E_i$ achieve the minimum value of the log discrepancy, hence the corresponding simplex $\Delta_J\subset \Delta_{\mathcal{X}}$ is an $n$-dimensional face of $Sk(X)\subset \Delta_{\mathcal{X}}$. 
Let $z_i$ be the local equations of the divisors $E_i$, and we can arrange $t=z_0\ldots z_n$ by the semistable SNC condition. Using the local trivialisation of $\mathcal{L}$, we can regard sections $s$ of $L$ as local meromorphic functions, and Taylor expand
\[
s=\sum_{ \beta\in \Z^{m+1} } c_\beta z^\beta, \quad c_\beta\in \C,
\]
where only finitely many terms contain negative exponents.
The monomial valuation of $s$ at any point $x\in \Delta_J\subset Sk(X)$ is
\[
-\log 
|s|(x)=val_x(s) = \min \{    \langle x,\beta \rangle |  c_\beta \neq 0  \}. 
\]
Consequently, on any given face $\Delta_J$, the function $\log |s|$ is the maximum of a finite collection of integral affine linear functions, and in particular is piecewise affine linear.

For any fixed $l\geq 1$, and any nonzero section $s$ of $L^{\otimes l}\to X_K^{an}$, 
the valuation functions $\log |s|$ has a uniform Lipschitz estimate independent of $s$ by Lemma \ref{lem:Lip}, so modulo adding an integer constant, there are only a finite number of possibilities for these integral affine linear functions. In particular, there is a finite union of codimension one rational walls $\mathcal{S}_l\subset Sk(X)$ which contains all the lower dimensional faces of $Sk(X)$, together with the corner locus of $\log |s|$ for all such  $s$. We denote the connected components of $Sk(X)\setminus \mathcal{S}_l$ as $O_a$, which are open polyhedral domains inside the $n$-dimensional open faces $\Delta_J^0= \text{Int}(\Delta_J)$.

On each $O_a$, the valuation functions $\log |\theta_\alpha^l |(x)=-\langle x, \beta_\alpha\rangle $ are affine linear, and correspondingly only one monomial term $c_{\alpha, \beta} z^{\beta_\alpha}$ dominates all the other terms in the Taylor expansion of $\theta^l_\alpha$.  We partition the exponents $\beta \in \Z^{n+1}$ into equivalence classes modulo integer multiples of $(1,1\ldots 1)\in \Z^{n+1}$.

\begin{lem}\cite[Lemma 3.6]{Litheta}
	\label{lem:valuativeindependence2} For fixed $l\geq 0$, 
	suppose $\{ \theta_\alpha^l \}$ is a valuative independent $K$-basis of $H^0(X_K, lL)$, then on any $O_a\subset Sk(X)\setminus \mathcal{S}_l$, the set of leading exponents $\{ \beta_\alpha\}$ are all distinct in $\Z^{n+1}/\Z(1,\ldots 1)$.

\end{lem}

\begin{proof}
Suppose for $\alpha, \alpha'$, their leading exponents $\beta, \beta'$ fall within the same equivalence class.  The Taylor expansion for $\theta_\alpha^l$ and $\theta_{\alpha'}^l$ have leading terms $c_{\alpha, \beta} z^\beta $ and $c_{\alpha', \beta'} z^{\beta'}$ respectively, where $  c_{\alpha, \beta} ,  c_{\alpha', \beta'} \in \C^*$ and   $\beta'- \beta= k(1,\ldots 1)$ for some $k\in \Z$. Thus
in the Taylor expansion for the linear combination $c_{\alpha', \beta'} \theta_\alpha^l - c_{\alpha, \beta} t^{-k} \theta_{\alpha'}^l$, the coefficient for the leading term $z^\beta$ cancels, hence at any point $x\in O_a$,
\[
val_x (    c_{\alpha', \beta'} \theta_\alpha - c_{\alpha, \beta} t^{-k} \theta_{\alpha'}^l    ) > \min( val_x(\theta_\alpha^l),  val_x(\theta_{\alpha'}^l) + val_x(  t^{-k} ) ),
\]
violating the valuative independence condition. 
\end{proof}

	\begin{rmk}
	Suppose we have two sets of valuative independent $K$-basis $\{  \theta_\alpha^l  \}$ and $\{ \tilde{\theta}_{\alpha'}^l \}$ for $H^0(X_K, lL)$. Denote $N(l)=\dim_K H^0(X_K, lL)$.
The relation between the two bases is described in  \cite[Remark 3.8]{Litheta}. The transition matrix $A\in GL(N(l), K)$ is defined by 
$
\theta_\alpha^l = \sum_{\alpha'} A_{\alpha \alpha'} \tilde{\theta}_{\alpha'}^l.
$
Up to multiplying $\tilde{\theta}_{\alpha'}^l$ by suitable powers of $t$, we can find a bijective matching between the index sets $\alpha, \alpha'$, such that $val_x(\theta_\alpha^l)= val_x(\tilde{\theta}_{\alpha'}^l)$ for all $x\in   Sk(X)$. Furthermore, there is an invertible diagonal matrix $\bar{A}_{\alpha \alpha'}\in GL(N(l), \C)$, such that the transition matrix $\bar{A}^{-1}A$ is a perturbation of the identity that does not change the leading order Taylor expansion terms of $\{  \theta_\alpha^l \}$.

We will sometimes refer to $\log |\theta_\alpha^l|: Sk(X)\to \R$ as \emph{tropical theta functions}.
	These are canonical up to additive constants (corresponding to multiplying $\theta_\alpha^l$ by powers of $t$).
	\end{rmk}

	\subsection{Cost function, c-transform}\label{sec:costfunction}

For each $l\geq 1$, let $\{  \theta_\alpha^l  \}$ be a valuative independent $K$-basis of $H^0(X_K, lL)$. 
For any given parameter $y\in Sk(X)$, we consider the set of functions on $Sk(X)$,
\[
\{      \frac{1}{l}  \log |\theta_\alpha^l | (x)-   \frac{1}{l}  \log |\theta_\alpha^l | (y) \}_{l,\alpha}.
\]
By construction, these functions take the value zero at the point $y\in Sk(X)$, and restrict to convex functions on the faces of $Sk(X)$. 
By the uniform Lipschitz estimate in Lemma \ref{lem:Lip}, this set is precompact in $C^0(Sk(X))$.

We define the set of \emph{limiting tropical theta functions} $B_y$ to be the set of all subsequential $C^0$-limits of these functions, as $l\to +\infty$. 
For any $p\in B_y$, we denote the corresponding function on $Sk(X)$ as $c(x,p; y)$. The function $c(x,p;y): Sk(X)\times B_y\to \R$ is called the \emph{cost function}. It inherits the uniform Lipschitz estimate
\begin{equation}\label{eqn:costfunctionLipschitz}
\sup_{p\in B} |c(x,p; y)- c(x',p; y)|\leq C|x-x'|
\end{equation}
for any $x,x'$ on the same face of $Sk(X)$. Hence the set of functions $B_y$ is compact in $C^0$-topology.
 On each $n$-dimensional face of $Sk(X)$, the cost function is convex in $x$. The dependence on $y$ is very mild by the following simple observation, and we will sometimes drop $y$ in the notation.

	\begin{lem}\label{lem:changey}
		Given any $y, y'\in Sk(X)$, there is a canonical homeomorphism between $B_y$ and $B_{y'}$, by sending the function $c(\cdot ,p;y): Sk(X)\to \R$ to $c(\cdot, p; y)-c(y', p; y): Sk(X)\to \R$. 
	\end{lem}

	Given the cost function $c(x,p)$, we can define the \emph{$c$-transform} for bounded functions $\phi: Sk(X)\to \R$ and $\psi: B_y\to \R$:
	\begin{equation}\label{eqn:ctransform}
		\psi^c(x)= \sup_{p\in B_y}c(x,p) - \psi(p),\quad
		\phi^c(p)= \sup_{x\in Sk(X)} c(x,p) -\phi(x).
	\end{equation}
	This can be viewed as a generalisation of the Legendre transform. Formal arguments imply that

	\begin{lem}\label{lem:Pc}
		For any $\phi\in L^\infty(Sk(X))$ and $\psi\in L^\infty(B_y)$, we have $(\phi^c)^c\leq \phi$ and $(\psi^c)^c \leq \psi$. For $\phi_1,\phi_2\in L^\infty(Sk(X))$, we have $\norm{\phi_1^c-\phi_2^c}_{L^\infty}\leq \norm{\phi_1-\phi_2}_{L^\infty}$, and likewise for functions in $L^\infty(B_y)$.

		Moreover, for $\phi, \psi$ in the image of the $c$-transform, the $c$-transform is involutive: $(\phi^c)^c=\phi$ and $(\psi^c)^c=\psi$. We shall denote this class of functions $\phi$ as $\mathcal{P}_c$. By Lemma \ref{lem:changey}, the function class $\mathcal{P}_c$ is independent of the choice of $y\in Sk(X)$.
	\end{lem}

	\begin{lem}\label{lem:uniformcontinuityPc}
		The functions $\phi\in \mathcal{P}_c$ have uniform Lipschitz continuity on $Sk(X)$, and are convex when restricted on any face of $Sk(X)$.
	\end{lem}

	\begin{proof}
		This follows formally from the uniform Lipschitz estimate (\ref{eqn:costfunctionLipschitz}), and the convexity of $c(x,p;y)$ in the $x$-variable on any face of $Sk(X)$.
	\end{proof}

	\begin{lem}\label{lem:ctransformmax}
	For any $\phi\in \mathcal{P}_c$, we have 
	\[
	\begin{cases}
	\phi(x)= \max_{p\in B_y} c(x,p)- \phi^c(p)
	\\
	\phi^c(p)= \max_{x\in Sk(X)} c(x,p)- \phi(x).
	\end{cases}
	\]

	\end{lem}

	\begin{proof}
	The continuity of $c(x,p)$ in $x$ follows from the Lipschitz bound, while the continuity in $p$ is tautological by the $C^0$-topology on $B_y$. Thus $\phi^c$ and $\phi=(\phi^c)^c$ are both continuous. The sup must be achieved by the compactness of $Sk(X)$ and $B_y$.
		\end{proof}

	\subsection{Comparison with NA semipositive potentials}\label{sec:Pcvspsh}

	The functions in $\mathcal{P}_c$ have unique \emph{maximal psh extensions} to semipositive potentials in $\text{CPSH}(X_K^{an},\mathcal{L})$.

	\begin{prop}\label{prop:maxextension}
		For any potential $\phi\in \mathcal{P}_c$ defined on $Sk(X)\subset X_K^{an}$, there is a unique $\phi\in  \text{CPSH}(X_K^{an},\mathcal{L})$, whose restriction to $Sk(X)\subset X_K^{an}$ agrees with the prescribed $\phi\in \mathcal{P}_c$, and moreover satisfies the 
	\emph{domination property}
		\begin{equation}\label{eqn:dominationproperty}
\phi(x)= \max\{ v(x) :v\in \text{CPSH}(X_K^{an},\mathcal{L}), v|_{Sk(X)}\leq \phi |_{Sk(X)}\},\quad \forall x\in X_K^{an}.
		\end{equation}
	\end{prop}

	\begin{proof}
		Let $\phi\in \mathcal{P}_c$, and let $\phi^c: B_y\to \R$ be its $c$-transform. We define $\tilde{\phi} \in  \text{CPSH}(X_K^{an},\mathcal{L})$ as the psh envelop of $\phi\circ r_{\mathcal{X}}\in C^0(X_K^{an})$ (\cf Section \ref{sec:envelop}), hence
		\[
		\tilde{\phi}(x)= \max\{ v(x) :v\in \text{CPSH}(X_K^{an},\mathcal{L}), v|_{Sk(X)}\leq \phi|_{Sk(X)}\}. 
		\]
		We need to show that $\tilde{\phi}(x)= \phi(x)$ for any $x\in Sk(X)$. Clearly $\tilde{\phi}(x)\leq \phi(x)$.

		Let $\epsilon>0$ be any given small constant. By Lemma \ref{lem:ctransformmax},
		for any $x\in Sk(X)$, we can find some $p\in B_y$, such that 
		\[
		\phi(x) = c(x,p;y)-\phi^c(p). 
		\]
		Since $c(\cdot ,p; y)$ is a subsequential $C^0$-limit, we can find some $\theta_\alpha^l$ with $l\gg 1$, such that for any $x\in Sk(X)$, 
		\[
		| \frac{1}{l}  \log |\theta_\alpha^l | (x)-   \frac{1}{l}  \log |\theta_\alpha^l | (y) - c(x,p;y) |< \epsilon
		\]
		We find some NA Fubini-Study potential $\phi_l$ on $(X_K^{an},L)$, such that 
		\[
		\phi_l = l^{-1} \log |\theta_\alpha^l|-  l^{-1} \log |\theta_\alpha^l | (y) - \phi^c(p)-\epsilon \quad \text{on $Sk(X)$}.
		\]
		Thus $\phi_l\leq \phi$ on $Sk(X)$, so the domination property of $\tilde{\phi}$ implies	
		$\tilde{\phi}(x)\geq \phi_l(x)$, whence for any $x\in Sk(X)$,
		\[
		\tilde{\phi}(x) \geq \phi_l(x)\geq   c(x,p)- \phi^c(p) -2\epsilon \geq \phi(x)-2\epsilon.
		\]
		Since $x,\epsilon$ are arbitrary, this proves that $\tilde{\phi}(x)\geq \phi(x)$ for any $x\in Sk(X)$. Hence $\tilde{\phi}= \phi$ on $Sk(X)$, and $\tilde{\phi}\in \text{CPSH}(X_K^{an},\mathcal{L})$ furnishes the maximal psh extension.
	\end{proof}

	In the converse direction,
	
	\begin{prop}\label{prop:maxextension2}
	Suppose $\phi \in \text{CPSH}(X_K^{an},\mathcal{L})$, then its restriction to $Sk(X)$ lies in $\mathcal{P}_c$.
	\end{prop}

	\begin{proof}
	Any $\phi\in  \text{CPSH}(X_K^{an},\mathcal{L})$ can be $C^0$-approximated by Fubini-Study potentials. For any large $l\gg 1$, and any section $s\in H^0(X_K, lL)$, upon writing $s=\sum a_\alpha \theta_\alpha^l$, by the valuative independence condition
	\[
	\log |s| (x)= \max_{\alpha: a_\alpha\neq 0} (\log |\theta_\alpha^l|(x)- val(a_\alpha)),\quad \forall x\in Sk(X).
	\]
	A general Fubini-Study potential is of the form
	\[
	\phi_{FS}=l^{-1}\max_i (  \log |s_i|- c_i)
	\]
	for  some $s_i\in H^0(X_K, lL)$, 
hence  on $Sk(X)$, the potential  $\phi_{FS}$  agrees with 
	\[
	\max_\alpha (l^{-1}\log |\theta_\alpha^l|(x)-   l^{-1}\log |\theta_\alpha^l|(y) -     c_\alpha)
	\]
	for suitable choices of $c_\alpha\in \R$.

By the equicontinuity estimate in  Lemma \ref{lem:Lip} and the compactness of $B_y$, there are some uniform constants $\epsilon_l$ depending only on $l$, such that for any $\theta_\alpha^l$, there is some $p\in B_y$ such that
	\[
	\sup_{x\in Sk(X)} | l^{-1} \log |\theta_\alpha^l |(x)-  l^{-1}\log |\theta_\alpha^l|(y) -c(x,p;y) |\leq \epsilon_l \to 0,\quad l\to +\infty.
	\]
Thus we have found a function $\varphi=   \max_p( c(x,p;y)- c_\alpha) $ satisfying 
$
	|\phi_{FS} - \varphi| \leq \epsilon_l.
$
	The function $\varphi$ lies in $\mathcal{P}_c$. The upshot is that on $Sk(X)$, any $\phi\in  \text{CPSH}(X_K^{an},\mathcal{L})$ can be $C^0$-approximated by functions in $\mathcal{P}_c$. But $\mathcal{P}_c$ is closed under $C^0$-topology, hence $\phi|_{Sk(X)}$ belongs to $\mathcal{P}_c$.
	\end{proof}

In summary, the functions in $\mathcal{P}_c$ can be identified with the NA semipositive potentials satisfying the domination property.

		\begin{cor}\label{cor:Pcconvex}
	 $\mathcal{P}_c$ is a convex subset  inside $C^0(Sk(X))$.
	\end{cor}

	\begin{proof}
	Let $\phi_1,\phi_2\in \mathcal{P}_c$, and $0\leq \lambda\leq 1$. By Prop. \ref{prop:maxextension}, we take the maximal  psh extension of $\phi_1, \phi_2$ to potentials in $\text{CPSH}(X_K^{an},\mathcal{L})$, denoted as $\tilde{\phi}_1, \tilde{\phi}_2$. Since $\text{CPSH}(X_K^{an},\mathcal{L})$ is convex, we see $\lambda \tilde{\phi}_1+(1-\lambda)\tilde{\phi}_2\in \text{CPSH}(X_K^{an},\mathcal{L})$. This function restricts to $\lambda \phi_1+(1-\lambda)\phi_2$, which belongs to $\mathcal{P}_c$ by Prop. \ref{prop:maxextension2}.
	\end{proof}

	\begin{cor}\label{cor:tropicalthetainPc}
The functions $l^{-1} \log |\theta_\alpha^l| : Sk(X)\to \R$ belong to $\mathcal{P}_c$ whenever $l\geq 1$.
	\end{cor}

	\begin{proof}
	The function $l^{-1} \log |\theta_\alpha^l| : Sk(X)\to \R$ can be extended to a Fubini-Study potential in $\text{CPSH}(X_K^{an},\mathcal{L})$, so the restriction to $Sk(X)$ is a function in $\mathcal{P}_c$ by Prop. \ref{prop:maxextension2}.
	\end{proof}

	\subsection{Gradient convex hull and Okounkov body}\label{sec:gradientconvexhullOkounkovbody}

The set of limiting tropical theta functions $B_y$ is  closely related to certain convex sets $\mathcal{C}_y, \Delta_y$.

Recall that for any convex function $u$ defined on an open convex domain in $\R^n$, the \emph{subgradient} at any point $x$ is the set 
\[
\nabla u(x)=\{     p\in \R^n:    u(x') \geq u(x)+ \langle x'-x, p\rangle , \forall x'     \}\subset \R^n.
\]
When this set is a singleton, we often identify it with the gradient vector.
Recall that $\mathcal{P}_c$ is defined independent of the parameter $y\in Sk(X)$.
Now for any $\phi\in \mathcal{P}_c$, its restriction to any $n$-dimensional face $\Delta_J\subset Sk(X)$ is a convex function by Lemma \ref{lem:uniformcontinuityPc}. Thus at any $x\in \Delta_J^0=\text{Int}(\Delta_J)$, we can define the \emph{gradient convex hull}
\begin{equation}
\mathcal{C}_x=  \{  p:    p\in \nabla \phi(x)  \text{ for some } \phi\in \mathcal{P}_c       \}  \subset T_y^* \Delta_J \simeq \R^n .
\end{equation}

\begin{lem}
$\mathcal{C}_x$ is a closed, bounded and convex subset of $\R^n$.
\end{lem}

\begin{proof}
The boundedness of $\mathcal{C}_x$ follows from the uniform Lipschitz estimate for $\phi\in \mathcal{P}_c$. If $p_i\in \nabla \phi_i$ for some $\phi_i\in \mathcal{P}_c$ with $i=1,2$, then any convex combination $\lambda p_1+ (1-\lambda) p_2$ for $0\leq \lambda\leq 1$ is a subgradient for 
$\lambda \phi_1+ (1-\lambda) \phi_2$, which belongs to $\mathcal{P}_c$. Thus $\mathcal{C}_x$ is convex.

The function class $\mathcal{P}_c$ is compact modulo an additive constant that does not affect the subgradients. Given any sequence $p_i\in \mathcal{C}_x\subset \R^n$ converging to $p\in \R^n$, we can find $\phi_i\in \mathcal{P}_c$ which are uniformly bounded, such that $p_i\in \nabla \phi_i(x)$. Then we can take a subsequential $C^0$-limit $\phi\in \mathcal{P}_c$, and we observe $p\in \nabla \phi(x)$. Thus $\mathcal{C}_x$ is closed. 
\end{proof}

	\begin{rmk}
		We have abused the same notation $p$ to denote gradient, and the elements in $B_y$. Their relationship will be more apparent in Section \ref{sec:Okounkovcost}.
	\end{rmk}

Our next goal is to introduce an Okounkov body, associated to a semigroup constructed from the gradient of tropical theta functions.

On any open $n$-dimensional faces $\Delta_J^0$, there is a natural integral affine structure.
Recall from Section \ref{sect:valuativeindependence2} that for any $l\geq 1$, there is a finite union $\mathcal{S}_l$ of codimension one walls including all the lower dimensional faces of $Sk(X)$, such that on each connected component $O_a\subset Sk(X)\setminus \mathcal{S}_l$, the tropical theta functions $\log |\theta_\alpha^l|= - \langle x, \beta_\alpha\rangle $ are \emph{integral affine linear} on $O_a$.

\begin{Notation}
	We say that a point $y$ on an $n$-dimensional open face $\Delta_J^0\subset Sk(X)$ is \emph{sufficiently irrational}, if $y\notin \cup_{l\geq 1} \mathcal{S}_l$. Then at the point $y$, all the tropical theta functions $\log |\theta_\alpha^l |$ have a unique gradient 
	\begin{equation}
		p(\alpha):= (\nabla \log |\theta_\alpha^l|)(y) \in \Z^n\subset   \R^n\simeq T_y^* \Delta_J^0.
	\end{equation}
	(As a caveat, the domain of affine linearity for $\log |\theta_\alpha^l|$ around $y$ may become smaller as $l\to +\infty$, and the infinite intersection may not be open.) We will build a semigroup $\Gamma_y\subset \Z^{n+1}$ from the data of tropical theta functions. As a set,
	\begin{equation}
	\Gamma_{y,l}= \bigcup_\alpha \{     (\nabla \log |\theta_\alpha^l|)(y)              \}\subset \Z^n, \quad \Gamma_y= \bigcup_{l\geq 0}  (\{  l             \}\times \Gamma_{y,l} ) \subset \N_{\geq 0}\times \Z^n.
	\end{equation}
	The \emph{Okounkov body} $\Delta_y\subset \R^n$ is the closure of the convex hull for $\bigcup_{l\geq 1}  l^{-1} \Gamma_{y,l}$. 
	We denote $\Delta_y^0=\text{Int}(\Delta_y)$. 
\end{Notation}

\begin{lem}  (Basic properties of $\Gamma_y$)\label{lem:Gammay}
\begin{enumerate}
	\item 	For any given $l$,	the gradients $p(\alpha)\in \Z^n$ are all distinct for different $\alpha$.

	\item   $\Gamma_{y,l}$ is the union of all $(\nabla \log |s| )(y)$ for all nonzero $s\in H^0(X_K, lL)$.

	\item     The subset $\Gamma_y$ is a \emph{semigroup} under addition.

	\item   The semigroup $\Gamma_y$ generates $\Z^{n+1}$ as a group.

	\item  Any $p\in \Gamma_{y,l}$ satisfies the uniform linear growth bound $|p|\leq Cl$ for some uniform constant independence of $l,y$. Consequently, the Okounkov body is bounded.
\end{enumerate}
\end{lem}

\begin{proof}
	The integral lattice $\Z^n$ inside $T_y^* \Delta_J^0$ can be identified as $\Z^{n+1}/\Z(1,\ldots 1)$, where the Taylor exponents $\beta_\alpha$  lie in $\Z^{n+1}$, and the exponent of $t=z_0\ldots z_n$ is $(1,\ldots 1)\in \Z^{n+1}$. The gradient $p(\alpha)=(\nabla \log |\theta_\alpha^l|)(y)\in \Z^n$  is then the image of $-\beta_\alpha$ inside the quotient $\Z^{n+1}/\Z(1,\ldots 1)$. These are all distinct by Lemma \ref{lem:valuativeindependence2}, which proves item 1.

Given any $s\in H^0(X_K, lL)$, we can expand
\[
s= \sum_\alpha a_\alpha \theta_\alpha^l,\quad a_\alpha\in K.
\]
Since $y$ is sufficiently irrational, we can find some domain of linearity $O_a\subset Sk(X)$ containing $y$ in its interior, where for some choice of $\alpha$,
\[
\log |s|= \log |\theta_\alpha^l |- val(a_\alpha),
\]
hence $(\nabla \log |s| )(y)= (\nabla \log |\theta_\alpha^l| )(y)\in \Gamma_{y,l}$. This proves item 2.

Now given $\theta_\alpha^l \in H^0(X_K, lL)$ and $\theta_{\alpha'}^{l'}\in H^0(X_K, l' L)$, 
using item 2, we have
\[
(\nabla \log |\theta_\alpha^l | )(y) +(\nabla \log |  \theta_{\alpha'}^{l'}   |) (y) =( \nabla \log |\theta_\alpha^l   \theta_{\alpha'}^{l'}| )(y) \in  \Gamma_{y,l+l'},
\]
so $\Gamma_y$ is closed under addition. This proves item 3.

Item 4 follows from item 2 and the relative ampleness of $L$, by the argument in \cite[Prop. 3.3]{BoucksomOkounkov}.

Item 5 follows from the uniform Lipschitz bound in Lemma \ref{lem:Lip}.
\end{proof}

The general theory of Okounkov bodies for semigroups then imply the following.

\begin{cor}\label{cor:semigroup}
Let $\mathcal{K}$ be any fixed compact  convex  subset of $\Delta_y^0$. 
\begin{enumerate}

	\item (Integer points) For all large $l\geq l_0(\mathcal{K})$, if $p\in \mathcal{K}\cap l^{-1} \Z^n$, then $lp\in \Gamma_{y,l}$.

		\item The subset $ \bigcup_{l\geq 0} \{    l        \} \times (\Gamma_{y,l}\cap l\mathcal{K}) \subset \Gamma_y$ is a sub-semigroup. It is contained inside a finitely generated sub-semigroup of $\Gamma_y$.

	\item The Lebesgue volume of the Okounkov body $\text{Vol}_{\R^n}(\Delta_y) = \frac{(L^n)}{n!}$.
\end{enumerate}

\end{cor}

\begin{proof}
Item 1 follows from the argument in  \cite[Lemma 2.3]{WittNystrom}.

For item 2, notice that the convexity of $\mathcal{K}$ implies that $\bigcup_{l\geq 0} \{    l        \} \times (\Gamma_{y,l}\cap l\mathcal{K})$ is closed under addition, hence defines a sub-semigroup. Now we can find a closed rational convex polyhedron $\mathcal{K}'$, such that $\mathcal{K}\subset \mathcal{K}'\subset \Delta_y^0$. Then
\[
 \bigcup_{l\geq 0} \{    l        \} \times (\Gamma_{y,l}\cap l\mathcal{K})  \subset  \bigcup_{l\geq 0} \{    l        \} \times (\Gamma_{y,l}\cap l\mathcal{K}') \subset \Gamma_y.
\]
But by item 1,  we have  $ \Gamma_{y,l}\cap l\mathcal{K}'= \Z^n \cap l\mathcal{K}'$   for large enough $l$, whence $ \bigcup_{l\geq 0} \{    l        \} \times (\Gamma_{y,l}\cap l\mathcal{K}')$ is a finitely generated sub-semigroup.

For item 3, we take an exhaustion of $\Delta_y^0$ by compact convex subsets $\mathcal{K}$, so that the points  in $l^{-1}\Z^n\cap (\Delta_y\setminus \mathcal{K})$ are negligible. Using item 1 and Riemann-Roch, the Lebesgue volume of the bounded convex body $\Delta_y$ is
\[
\text{Vol}_{\R^n}(\Delta_y) = \lim_{l\to +\infty} \frac{ |\Delta_y^0\cap l^{-1}\Z^n| }{ l^n  }= \lim_{l\to +\infty}  \frac{ |\Delta_{y,l}| }{ l^n  }=  \lim_{l\to +\infty} \frac{ \dim_K H^0(X_K, lL)}{l^n}= \frac{(L^n)}{n!}.
\]

\end{proof}

We now compare the \emph{gradient convex hull} with the \emph{Okounkov body}.

\begin{lem}\label{lem:gradientconvexhullOkounkovbody}
	Suppose $y\in Sk(X)$ is sufficiently irrational. Then for any $l\geq 1$, 
	\begin{enumerate}

		\item  The Okounkov body $\Delta_y\subset \mathcal{C}_y$.

		\item  Suppose there is some $\phi\in \mathcal{P}_c$ such that $\nabla\phi(y)$ consists of a single point $p$, then $p\in \Delta_y$.
	\end{enumerate}

\end{lem}

	\begin{proof}

	For $l\geq 1$, the functions $l^{-1} \log |\theta_\alpha^l| : Sk(X)\to \R$ belong to $\mathcal{P}_c$, hence their gradients lie inside $\mathcal{C}_y$. This proves $\Delta_y\subset \mathcal{C}_y$.

	By the proof of Prop. \ref{prop:maxextension2}, any Fubini-Study potential restricted to $Sk(X)$ is of the form 
		\[
	\phi_{FS}= \max_\alpha (l^{-1}\log |\theta_\alpha^l|(x)-   l^{-1}\log |\theta_\alpha^l|(y) -     c_\alpha),
	\]
	hence its gradient lies in the convex hull $\Delta_y$. Now any $\phi\in  \text{CPSH}(X_K^{an},\mathcal{L})$ can be $C^0$-approximated by Fubini-Study potentials, hence any function  $\phi\in \mathcal{P}_c$ can be $C^0$ approximated by a sequence $\phi_{FS,l}$ as $l\to +\infty$. Then there is a sequence of gradients $p_l\in \Delta_y$, such that
	\[
	\phi_{FS, l} (x)- \phi_{FS,l}(y)\geq \langle p_l, x-y\rangle, \quad \forall x\in \Delta_J^0.
	\]
	Since $p_l\in \Delta_y\subset \mathcal{C}_y$ are uniformly bounded, they converge to some $p_\infty\in \Delta_y$.
	Taking the $C^0$-limit,
	\[
	\phi(x)- \phi(y) \geq \langle p_\infty, x-y\rangle, \quad \forall x\in \Delta_J^0,
	\]
	hence $p_\infty\in \nabla \phi(y)$. If $\nabla \phi(y)$ consists of only one point $p$, then $p=p_\infty\in \Delta_y$.
	\end{proof}

	\begin{rmk}
	It is a curious question whether $\mathcal{C}_y=\Delta_y$ for sufficiently irrational $y$. The $C^0$-approximation argument above is unfortunately not strong enough to prove this equality. As a toy case, on the real line $\R$ one can find a sequence of convex functions $\max( |x|, \epsilon)$ for $\epsilon>0$, which all have zero gradient at the origin, but the limit function $|x|$ has many subgradients. 
	\end{rmk}
	
	\begin{rmk}
	It would be interesting to give a more birational geometric interpretation for the gradient convex hull and the Okounkov body.
	\end{rmk}
	

	\subsection{Okounkov body and cost function}\label{sec:Okounkovcost}

	Let $y\in Sk(X)$ be a sufficiently irrational point.
	The Okounkov body can be used to produce limiting tropical theta functions in $B_y$.
Recall that a real valued function on a semigroup $F:\Gamma\to \R$ is called \emph{subadditive} if 
$
		 F (a) + F (b) \geq F (a + b).
$
	For the semigroup $\Gamma_y$, by item 1 in Lemma \ref{lem:Gammay}, for each given $l\geq 1$, there is a natural bijection between the indexing set for $\{   \theta_\alpha^l \}$, and the gradients $p(\alpha): =( \nabla \log |\theta_\alpha^l | )(y)\in \Z^n$. Now for any $x\in Sk(X)$, we can define the function  
	\[
	F_x:  \Gamma_y\mapsto \R,\quad 	(l, p(\alpha)) \mapsto \log |\theta_\alpha^l|(x)- \log |\theta_\alpha|(y).
	\]

	\begin{lem}
		For any choice of $x\in Sk(X)$, the function $F_x$ is subadditive.
\end{lem}

\begin{proof}
We can expand the product
\[
\theta_\alpha^l \theta_{\alpha'}^{l'}= \sum_\gamma a_\gamma \theta_\gamma^{l+l'},\quad a_\gamma\in K.
\]
By the valuative independence condition, 
\[
\log |\theta_\alpha^l |(x) + \log | \theta_{\alpha'}^{l'}|(x)= \log |\theta_\alpha^l \theta_{\alpha'}^{l'}|(x) =\max_\gamma \log  |\theta_\gamma^{l+l'}|(x)- val(a_\gamma).
\]
We pick $\gamma$ so that $p(\alpha)+ p(\alpha')= p(\gamma)$. Then  on the domain of affine linearity $O_a\subset Sk(X)\setminus \mathcal{S}_{l+l'}$ containing the point $y$, the dominant term in the  expansion of $\theta_\alpha^l \theta_{\alpha'}^{l'}$ has gradient $p(\gamma)$, so that on $O_a$ we have
\[
\log |\theta_\alpha^l |+  \log | \theta_{\alpha'}^{l'}|= \log |\theta_\alpha^l \theta_{\alpha'}^{l'}|= \log |\theta_\gamma^{l+l'}|- val(a_\gamma).
\]
In particular,
\[
\log |\theta_\alpha^l |(y)+  \log | \theta_{\alpha'}^{l'}|(y)= \log |\theta_\gamma^{l+l'}|(y)- val(a_\gamma),
\]
\[
\log |\theta_\alpha^l |(x)+  \log | \theta_{\alpha'}^{l'}|(x)\geq  \log |\theta_\gamma^{l+l'}|(x)- val(a_\gamma),
\]
hence
\[
F_x(l, p(\alpha))+ F_x(l', p(\alpha')) \geq F_x(l+l',   p(\alpha)+p(\alpha'))
\]
which means $F_x$ is subadditive.
\end{proof}

Given any $p\in \Delta_y^0$, 
the general theory for subadditive functions on semigroups allows us to define a function $c(\cdot, p;y)$ in the class $B_y$.

\begin{prop} (Limiting tropical theta functions I)\label{prop:limitingtropicalthetafromp} 	Let $y\in Sk(X)$ be a sufficiently irrational point.
\begin{enumerate}
	\item  Given $p\in \Delta_y^0$, for any sequence $(l, p(\alpha))\in \Gamma_{y,l}$ such that $l^{-1}p(\alpha)\to p$ as $l\to +\infty$, there is a well defined function in $B_y$,
	\[
	c(x,p;y)= \lim_{  l\to +\infty,   l^{-1}p(\alpha)\to p  }  l^{-1}(\log |\theta_\alpha^l|(x)- \log |\theta_\alpha^l|(y) ).
	\]

	\item As $p\in \Delta_y^0$ varies, the function $c(x,p;y)$ is convex in $p\in \Delta_y^0$.
	
	\item  For $p$ inside any given compact convex subset $\mathcal{K}\subset \Delta_y^0$, then $c(x,p;y)$ is a bounded Lipschitz function in $p$.

	\item  Inside any given compact subset $\mathcal{K}\subset \Delta_y^0$, there is a uniform convergence as $l\to +\infty$,
	\[
	\sup_{ x\in Sk(X) , l^{-1}p(\alpha) \in \mathcal{K} } |    l^{-1}(\log |\theta_\alpha^l|(x)- \log |\theta_\alpha^l|(y) )- c(x,  l^{-1}p(\alpha);y) | \leq \epsilon_l\to 0  .
	\]
\end{enumerate}
\end{prop}

	\begin{proof}
		We divide the proof into a few steps.
		
		\begin{enumerate}
			\item 	As in Section \ref{sec:costfunction}, there is always some subsequential $C^0$-limit function for $  l^{-1}(\log |\theta_\alpha^l|(x)- \log |\theta_\alpha|(y) )$, which defines an element of $B_y$. We note
			\[
			l^{-1}	F_x(l,p(\alpha))=  l^{-1}(\log |\theta_\alpha^l|(x)- \log |\theta_\alpha^l |(y) ).
			\] 
			Since $|F_x|\leq Cl$, and $F_x$ is subadditive on $\Gamma_y$, by  \cite[Thm. 3.1]{WittNystrom}  the sequence $l^{-1}F(l, p(\alpha))$ converges to a unique limit, depending on $p$ but not on the details of the subsequence. Thus the function $c(x,p;y)$ is well defined.
			
			\item 		Furthermore, by passing the subadditive property to the limit, we see $c(x,p;y)$ is convex in $p\in \Delta_y^0$ \cite[Thm. 3.1]{WittNystrom}.
			
			\item 
			The function $c(x,p;y)$ is bounded and convex in $p\in \Delta_y^0$, hence it is Lipschitz on any compact convex subset of $\Delta_y^0$. 
			
			\item If the uniform convergence in item 4 fails, then we can extract some subsequence with $l\to +\infty$, and $l^{-1} \alpha(p)\in \mathcal{K}$, $x\in Sk(X)$ depending on $l$, such that  
			\[
			|  l^{-1}(\log |\theta_\alpha^l|(x)- \log |\theta_\alpha^l |(y) )- c(x, l^{-1} p(\alpha);y) | \geq C^{-1}.
			\]
			By the compactness of $Sk(X)$ and $\mathcal{K}$, we may assume $x$ converges to $x_0\in Sk(X)$, and $l^{-1}p(\alpha)$ converges to $p$. By the equicontinuity in $x$, for $l\gg 1$,
			\[
			l^{-1}(\log |\theta_\alpha^l|(x_0)- \log |\theta_\alpha^l |(y) )- c(x_0,  l^{-1}p(\alpha);y) | \geq \frac{1}{2C}.
			\]
			But $ l^{-1}(\log |\theta_\alpha^l|(x_0)- \log |\theta_\alpha^l |(y) )\to c(x_0,p;y)$ by item 1, and
			\[
			|c(x_0, p ; y)- c(x_0, l^{-1} p(\alpha); y)|\leq C|p-l^{-1}p(\alpha)|\to 0\] 
			by item 3, so we get a contradiction. 
		\end{enumerate}
	\end{proof}

The limiting tropical theta functions arising from $p\in \Delta_y^0$ enjoy some further nice properties, which makes the $c$-transform resemble the classical Legendre transform.

	\begin{prop}\label{prop:limitingtropicalthetafromp2} 
	(Limiting tropical theta functions II) 	Let $y\in Sk(X)$ be a sufficiently irrational point, and let  $\mathcal{K}\subset \Delta_y^0$ be a fixed closed convex subset. 
	\begin{enumerate}
		\item (Domain of affine linearity) There is some open convex polyhedral neighbourhood $ y\in U$  depending on $\mathcal{K}$, contained in the open $n$-dimensional face $\Delta_J^0\subset Sk(X)$, such that for any $p\in \mathcal{K}\subset \Delta_y^0\subset \R^n$, the cost function is affine linear on $U$,
		\[
		c(x,p;y)= \langle p, x-y\rangle ,\quad \forall x\in U.
		\]

		\item In particular, for any $p\in \Delta_y^0$, the cost function $c(x,p;y)$ has a unique gradient $p$ at the point $y$.

\item   Suppose $\tilde{c}(\cdot ,p;y)\in \text{CPSH}(X_K^{an},\mathcal{L})$ is the maximal psh  extension for $c(\cdot,p;y)\in \mathcal{P}_c$ in the sense of Prop. \ref{prop:maxextension},
\[
\tilde{c}(x,p;y)= \max\{ v(x) :v\in \text{CPSH}(X_K^{an},\mathcal{L}), v|_{Sk(X)}\leq c(\cdot,p;y) \},\quad \forall x\in X_K^{an}.
\]
Then $\tilde{c}(\cdot, p;y)$ factors through the retraction map $r_\mathcal{X}: X_K^{an}\to \Delta_{\mathcal{X}}$ over the open set $U\subset  Sk(X)\subset \Delta_{\mathcal{X}}$, namely
\[
\tilde{c}(x, p;y)= c(   r_{\mathcal{X}}(x), p; y  ),\quad \forall x\in r_{\mathcal{X}}^{-1}(U).
\]

	\end{enumerate}
	\end{prop}

	\begin{proof}
		We divide the proof into a few steps.
		\begin{enumerate}
			\item 
	
	By item 2 in Cor. \ref{cor:semigroup}, we can find a finitely generated sub-semigroup for $\Gamma_y$ containing $ \bigcup_{l\geq 0} \{    l        \} \times (\Gamma_{y,l}\cap l\mathcal{K}) \subset \Gamma_y$.
	Each $\log |\theta_{\alpha_i}^{l_i}|$ is affine linear on some open polyhedron $O_a$ containing $y$. We take $U$ to be an open convex polyhedral neighbourhood of $y$ contained in the finite intersection of these $O_a$.

	For any $l\geq  1$, any $(l,   p(\alpha))\in \Gamma_{y,l}\cap l \mathcal{K}$ can be written as 
	\[
	(l,  p(\alpha))=  \sum_i k_i ( l_i, p(\alpha_i)),\quad k_i\in \Z_{\geq 0},\quad ( l_i, p(\alpha_i))\in \Gamma_{y,l_i}.
	\]
	We consider the  linear expansion of the product
	\begin{equation}\label{eqn:linearexpansionproduct}
	\prod_i (\theta_{\alpha_i}^{l_i} )^{\otimes k_i}= \sum_\gamma  a_\gamma  \theta_\gamma^l    \in H^0(X_K, lL),\quad a_\gamma\in K.
	\end{equation}
	Thus on $Sk(X)$, the valuative independence condition implies
	\begin{equation}\label{eqn:linearexpansionproduct2}
	\sum_i k_i \log |\theta_{\alpha_i}^{l_i} |=  \log | 	\prod_i (\theta_{\alpha_i}^{l_i} )^{\otimes k_i}|= \max_\gamma \log |\theta_\gamma^l |- val(a_\gamma).
	\end{equation}
	In particular,
	\begin{equation*}
		\sum_i k_i \log |\theta_{\alpha_i}^{l_i} |\geq \log |\theta_\alpha^l |- val(a_\gamma).
	\end{equation*}
	However, since $y$ is a sufficiently irrational point, there is a unique dominant term in the linear expansion (\ref{eqn:linearexpansionproduct}) on some smalll neighbourhood of $y$. The exponent of this term is determined by the gradient
	\[
(\nabla  \log 	|\prod_i (\theta_{\alpha_i}^{l_i} )^{\otimes k_i} |)(y)= \sum_i k_i (\nabla  \log 	| \theta_{\alpha_i}^{l_i} |)(y)=  \sum_i k_i p(\alpha_i)= p(\alpha)  .
	\]
	Thus in this small neighbourhood, the maximum in (\ref{eqn:linearexpansionproduct2}) is achieved by $\gamma=\alpha$. In particular 
	\[
		\sum_i k_i \log |\theta_{\alpha_i}^{l_i} |(y)=  \log |\theta_\alpha^l | (y)- val(a_\alpha).
	\]
		Combining the above, for any $x\in Sk(X)$,
\[
	\sum_i k_i  \left(\log |\theta_{\alpha_i}^{l_i} | (x) -  \log |\theta_{\alpha_i}^{l_i} | (y) \right)  \geq \log |\theta_\alpha^l |(x)- \log |\theta_\alpha^l | (y).
\]
On the open set $U$, all the finitely many functions $\log |\theta_{\alpha_i}^{l_i}|$ are affine linear, 
\[
\log |\theta_{\alpha_i}^{l_i} | (x) -  \log |\theta_{\alpha_i}^{l_i} | (y) = \langle p(\alpha_i), x-y\rangle,
\]
hence
\[
\begin{split}
&  	   \langle p(\alpha), x-y\rangle=  \sum_i k_i  \langle p(\alpha_i), x-y\rangle
	\\
	=& 	\sum_i k_i  \left(\log |\theta_{\alpha_i}^{l_i} | (x) -  \log |\theta_{\alpha_i}^{l_i} | (y) \right)  \geq \log |\theta_\alpha^l |(x)- \log |\theta_\alpha^l | (y).
\end{split}
\]
But the convexity of $\log |\theta_\alpha^l |$ gives the reverse inequality
\[
 \log |\theta_\alpha^l |(x)- \log |\theta_\alpha^l | (y) \geq  \langle p(\alpha), x-y\rangle.
\]
We conclude that on the open neighbourhood $y\in U$, \emph{the tropical theta function is affine linear},
\begin{equation}\label{eqn:tropicalthetaaffinelinear}
 \log |\theta_\alpha^l |(x)- \log |\theta_\alpha^l | (y) = \langle p(\alpha), x-y\rangle,
\end{equation}
whenever $l^{-1}\alpha(p)\in \mathcal{K}$. The key point here is that the set $U$ is fixed independent of $l,\alpha$, but only depends on $\mathcal{K}$.

Since $\mathcal{K}$ is a closed convex subset of $\Delta_y^0$,
 for any $p\in \mathcal{K}$, we can find a sequence of $(l, p(\alpha))\in l \mathcal{K}\cap \Z^{n+1}$ such that $ l^{-1} p(\alpha)\to p\in \mathcal{K}$ as $l\to +\infty$. We note that for $l\gg 1$, then $
 l \mathcal{K}\cap \Z^n= \Gamma_{y,l}\cap l\mathcal{K}$ by item 1 in Cor. \ref{cor:semigroup}. Thus
 \[
 \begin{split}
 c(x,p; y)= & \lim_{l\to +\infty, l^{-1}p(\alpha)\to p}  l^{-1} ( \log |\theta_\alpha^l |(x)- \log |\theta_\alpha^l | (y) )
 \\
 =& \lim_{l\to +\infty, l^{-1}p(\alpha)\to p}  l^{-1}  \langle p(\alpha), x-y\rangle
 \\
 = & \langle  p, x-y\rangle
 \end{split}
 \]
for all  $x\in U$.  This proves item 1.

\item 	Item 2 clearly follows from item 1.

\item

We return to  $\theta_\alpha^l$ for $l^{-1}p(\alpha)\in \mathcal{K}$ and $l\gg 1$. Near the depth $(n+1)$ intersection stratum $E_J$ corresponding to the $n$-dimensional face $\Delta_J\subset Sk(X)$, we can  Taylor expand 
\[
\theta_\alpha^l= \sum_\gamma c_\gamma  z^\gamma,\quad c_\gamma \in \C.
\]
Since $U$ is an open convex polyhedral subset of $\Delta_J^0\subset Sk(X)$, such that $\log |\theta_\alpha^l |$ is affine linear on $U$ with gradient $p(\alpha)$ by (\ref{eqn:tropicalthetaaffinelinear}), we can 
find a unique leading exponent $\beta_\alpha$ in the Taylor expansion, and it satisfies $p(\alpha)=- \beta_\alpha \in \Z^{n+1}/\Z(1,\ldots 1)\simeq \Z^n$. Upon writing
\[
\theta_\alpha^l= c_{\alpha, \beta} z^{\beta_\alpha}+ \sum_{\gamma\neq \beta_\alpha}  c_\gamma  z^\gamma.
\]
then
\[
\log |  \sum_{\gamma\neq \beta_\alpha}  c_\gamma  z^\gamma | (x)<  \log |\theta_\alpha^l | (y) + \langle p(\alpha), x-y\rangle,  \quad \forall x\in U\subset Sk(X),
\]
hence
\[
\log |  \sum_{\gamma\neq \beta_\alpha}  c_\gamma  z^\gamma | (x)<  \log |\theta_\alpha^l | (y) + \langle p(\alpha), x-y\rangle,  \quad \forall x\in r_\mathcal{X}^{-1}(U).
\]
The ultrametric property then implies that for any $x\in r_\mathcal{X}^{-1}(U)$,
\begin{equation}
\log |\theta_\alpha^l| (x)-  \log |\theta_\alpha^l| (y) = \log |c_{\alpha,\beta} z^{\beta_\alpha} |(x)- \log |\theta_\alpha^l| (y) = \langle p(\alpha), r_\mathcal{X}(x)- y\rangle .
\end{equation}

For any given $\epsilon>0$, similar to the proof of Prop. \ref{prop:maxextension}, we find some $\theta_\alpha^l$ with $l\gg 1$ and $l^{-1}p(\alpha)\in \mathcal{K}$, such that
\[
\sup_{x\in Sk(X)}| \frac{1}{l}  \log |\theta_\alpha^l | (x)-   \frac{1}{l}  \log |\theta_\alpha^l | (y) - c(x,p;y) |< \epsilon.
\]
Then the domination property of $\tilde{c}(x,p;y)$ implies that for any $x\in X_K^{an}$,
\[
\tilde{c}(x,p;y) \geq \frac{1}{l}  \log |\theta_\alpha^l | (x)-   \frac{1}{l}  \log |\theta_\alpha^l | (y) -  \epsilon.
\]
Taking the limit as $l\to +\infty$, over $x\in r_\mathcal{X}^{-1}(U)$, we have
\[
\tilde{c}(x,p;y) \geq  \lim_{l\to +\infty, l^{-1}p(\alpha)\to p} l^{-1}\langle p(\alpha), r_\mathcal{X}(x)- y\rangle - \epsilon = \langle p , r_\mathcal{X}(x)- y\rangle - \epsilon .
\]
Since $\epsilon>0$ is arbitrary, we obtain
\begin{equation}
\tilde{c}(x,p;y) \geq  \langle p , r_\mathcal{X}(x)- y\rangle ,\quad \forall x\in r_\mathcal{X}^{-1}(U).
\end{equation}

In the reverse direction,  recall  that any $\phi\in \text{CPSH}(X_K^{an},\mathcal{L})$ satisfies $\phi\leq \phi\circ r_\mathcal{X}$, hence over $x\in r_\mathcal{X}^{-1}(U)$
\[
\tilde{c}(x, p;y) \leq c(  r_{\mathcal{X}}  (x), p;y  )= \langle p,  r_\mathcal{X}(x)-  y\rangle.
\]
	Combining the above shows the equality
	\[
	\tilde{c}(x, p;y) = c(  r_{\mathcal{X}}  (x), p;y  )= \langle p,  r_\mathcal{X}(x)-  y\rangle.
	\]

		\end{enumerate}
	\end{proof}

		\begin{rmk}
		As a caveat, the convex domain $U$ of affine linearity depends on $\mathcal{K}$, and may shrink to a point when $\mathcal{K}$ increases to cover $\Delta_y^0$.

	\end{rmk}

We can naturally embed $\Delta_y^0$ as an open subset of $B_y$.

	\begin{prop}\label{prop:homeotoimage}(Embedding of $\Delta_y^0$)
	The assignment 
	\[
	\Delta_y^0\to B_y,\quad p\mapsto c(\cdot, p; y)
	\]
	is a homeomorphism onto an open subset of $B_y$. 
	\end{prop}

\begin{proof}
	We divide the proof into a few steps.
	\begin{enumerate}
		\item By Prop. \ref{prop:limitingtropicalthetafromp}, this assignment is well defined and depends continuously on $p\in \Delta_y^0$. 
		
		\item We claim the image of this assignment is an open subset in $B_y$. Any function $c(x,p';y)$ in $ B_y$ is a $C^0$-subsequential limit for $  l^{-1}(\log |\theta_\alpha^l|(x)- \log |\theta_\alpha^l |(y) )$  as $l\to +\infty$. After passing to a further subsequence, without loss $l^{-1} p(\alpha)$ converges to some $p_\infty\in \Delta_y$. If $p_\infty\in \Delta_y^0$, then $c(x,p';y)= c(x, p_\infty; y)$ lies in the image. Otherwise $p_\infty\in \partial \Delta_y$. But since $  l^{-1}(\log |\theta_\alpha^l|(x)- \log |\theta_\alpha^l |(y) )$ has gradient $l^{-1}p(\alpha)$ at $y$, upon passing to the $C^0$-limit for this sequence of convex functions, we deduce that $p_\infty$ belongs to the subgradient set of $c(x,p';y)$ at the point $x=y$.

		The upshot is that any function $c(x,p';y)$ in $ B_y$ lies in the image, if and only if its gradient at $y$ is contained in $\Delta_y^0$. This is an open property under the $C^0$-topology for uniformly Lipschitz convex functions, so the image is indeed open.

		\item By Prop. \ref{prop:limitingtropicalthetafromp2}, the inverse map from the image back to $\Delta_y^0$, is given by taking the gradient at the point $y$, and this inverse is continuous with respect to the $C^0$-topology on $B_y$, because for a sequence of convex functions with uniform Lipschitz bounds, the limit of their gradients at $y$ is a subgradient of the limiting convex function.
	\end{enumerate}
\end{proof}

	\begin{rmk}
	From the definition of $B_y$ as a subset of $C^0(Sk(X))$, it is only a compact topological space under the $C^0$-topology on functions, but has no a priori linear structure. On the other hand, $\Delta_y^0$ is an open convex set, so one can take linear combinations. By Lemma \ref{lem:changey}, one can canonically identify the $B_y$ for different choices of sufficiently irrational $y\in Sk(X)$.	It is curious what is the relationship between different $\Delta_y^0\subset B_y$ when $y$ varies. To answer this question, one presumably needs more information  on the cost function $c(x,p;y)$.
	\end{rmk}

	\begin{rmk}
		As a caveat, the inclusion map $\Delta_y^0\to B_y$ need not extend to a homeomorphism between the compact sets $\Delta_y\to B_y$. First, due to the ignorance of the uniform Lipschitz estimate in $p\in \Delta_y$ near the boundary $\partial \Delta_y$, we do not know if item 1 in Prop. \ref{prop:limitingtropicalthetafromp} still holds for $p\in \partial \Delta_y$. Second, some functions in $B_y$ (but not in the image of $\Delta_y^0$) may have non-unique subgradients at  $y\in Sk(X)$. We expect $B_y$ is in fact homeomorphic to the essential skeleton of some mirror CY family, which is topologically $S^n$ in the case of strict CY manifolds (See \cite[Sections 5,6]{Litheta}). In contrast, $\Delta_y$ is a convex set, hence contractible.
	\end{rmk}

	\begin{cor}
 Let $y\in Sk(X)$ be a sufficiently irrational point. Suppose a given function $\phi \in \mathcal{P}_c$ has a unique gradient  $\nabla \phi(y)=p$, and $p\in \Delta_y^0$. Then 
 $\phi(y)=  -\phi^c(p)$, and $\phi(x)\geq \phi(y)+ c(x,p;y)$ for any $x\in Sk(X)$.

	\end{cor}

	\begin{proof}
	By Lemma \ref{lem:ctransformmax}, there must be some  $p'\in B_y$ such that \[
	\phi(y)= c(y,p'; y)- \phi^c(p')= -\phi^c(p').\] By the proof of Prop. \ref{prop:homeotoimage} item 2,  because the gradient $p=\nabla \phi(y)$ is contained in $\Delta_y^0$, the function $c(x,p';y)$ in $B_y$ must lie in the embedding image of $\Delta_y^0$, and $p=p'$. This shows $\phi(y)= -\phi^c(p)$. 
	
	Now the definition of $c$-transform implies
	\[
	\phi(x)\geq c(x,p;y)- \phi^c(p)= c(x,p;y) + \phi(y),
	\]
for any $x\in Sk(X)$.
	\end{proof}

	The interpretation is that if we prescribe the value $\phi(y)$ and the gradient $p$ at $y$, then the function $ \phi(y)+ c(x,p;y)$ minimises among all $\phi\in \mathcal{P}_c$ subject to this prescription. Intuitively $c(x,p;y)$ plays the role of the linear functions in classical convex function theory.

	\subsection{A factorisation criterion for the retraction map}\label{sec:factorisation}

	By Prop. \ref{prop:maxextension}, \ref{prop:maxextension2}, we identify the functions in $\mathcal{P}_c$ with the potentials in $\text{CPSH}(X_K^{an},\mathcal{L})$ satisfying the domination property. We shall prove a criterion for a potential to factorise under the retraction map $r_\mathcal{X}: X_K^{an}\to \Delta_{\mathcal{X}}$ over some open subset of $Sk(X)$.

\begin{prop}\label{prop:factorisation}
Suppose  $\phi \in \text{CPSH}(X_K^{an},\mathcal{L})$ satisfies the domination property (\ref{eqn:dominationproperty}). 
Suppose $y\in Sk(X)$ is a sufficiently irrational point, such that its subgradient set $\nabla \phi(y)\subset \mathcal{C}_y\subset \R^n$
is contained in the interior of the Okounkov body $\Delta_y^0$. Then there exists some open neighbourhood $y\in U_y\subset  Sk(X)$, such that  $\phi= \phi\circ r_{\mathcal{X}}$  holds on  $r_\mathcal{X}^{-1}(U_y)$.
\end{prop}

	\begin{proof}
		We divide the proof into a few steps.
		
	\begin{enumerate}
		\item  We take a closed convex set $\mathcal{K}\subset \Delta_y^0$, which contains the subgradient set $\nabla \phi(y)$ in its interior. (Here we do not need to assume $\nabla \phi(y)$ is a singleton). We can find some open convex polyhedral neighbourhood $y\in U$ depending on $\mathcal{K}$, such that the conclusions of Prop. \ref{prop:limitingtropicalthetafromp2} hold, and in particular
		\[
		c(x,p;y)= \langle p, x-y\rangle,\quad \forall x\in U, \quad \forall p\in \mathcal{K}.
		\]
		Since $\nabla \phi(y)\subset \text{Int}(\mathcal{K})$, and $\phi$ is a convex function on $U$, we can find an open neighbourhood $y\in U_y\subset U$, such that 
		\[
		\nabla \phi(x)\subset \text{Int}(\mathcal{K}),\quad \forall x\in U_y.
		\]

		\item We consider any $y'\in U_y$, which may not be sufficiently irrational. By Lemma \ref{lem:ctransformmax}, there is some $p'\in B_y$, such that 
		\[
		\phi(y')= c(y',p';y) - \phi^c(p').
		\]
	By the definition of $B_y$, the function $c(\cdot, p'; y)$ is a $C^0$-limit for some sequence
		\[
		   \frac{1}{l}  \log |\theta_\alpha^l | (x)-   \frac{1}{l}  \log |\theta_\alpha^l | (y).
		\]
		We focus on  large $l\gg 1$, then Cor. \ref{cor:semigroup} item 1 implies $\Gamma_{y,l}\cap l\mathcal{K}= \Z^n \cap  l\mathcal{K}$.

	\item 
			Recall that if $p(\alpha)= (\nabla \log |\theta_\alpha^l|)(y)\in \Gamma_{y,l}\cap l \mathcal{K}=   \Z^n \cap  l\mathcal{K}  $, then by (\ref{eqn:tropicalthetaaffinelinear}),
		\[
		\log |\theta_\alpha^l |(x)- \log |\theta_\alpha^l | (y) = \langle p(\alpha), x-y\rangle, \forall x\in U_y,
		\]
		and in particular $(\nabla 	\log |\theta_\alpha^l | )(y')= p(\alpha) \in \Z^n \cap l \mathcal{K}$.
		This construction exhausts all the $\theta_\alpha^l$ such that 
		$(\nabla 	\log |\theta_\alpha^l | )(y')\in \Z^n \cap l \mathcal{K}$.

\item

			We claim that the function $c(\cdot, p'; y)$ in fact lies in the image of the embedding $\mathcal{K} \subset \Delta_y^0 \to B_y$ (\cf Prop. \ref{prop:homeotoimage}).

	In the first case,	suppose there is a subsequence with $p(\alpha)  \in \Gamma_{y,l}\cap l \mathcal{K}= \Z^n \cap  l\mathcal{K}$, then by taking the limit $l\to+\infty$, the claim would follow.

Otherwise, $p(\alpha)\notin \Z^n\cap l\mathcal{K}$ for all large enough $l$.
As a caveat, here we allow the possibility for $y'$ to lie on some codimension one wall in $\mathcal{S}_l$, in which case $\log |\theta_\alpha^l | $ may have non-unique subgradient at $y'$.
 Nontheless, if all its subgradients at $y'$ lie in $\Z^n\cap l\mathcal{K}$, then by the 
 distinctness of the leading order exponent (\cf Lemma \ref{lem:valuativeindependence2}), our $\theta_\alpha^l$ must coincide with one of those listed in Step 3, hence falls within the first case, contradiction.

Thus for $l\gg 1$, we 
 can find $p_l\in \nabla \log |\theta_\alpha^l | (y')$ such that $p_l\in  \Z^n\cap l (\Delta_y\setminus    \mathcal{K})$. Thus $l^{-1}p_l\in  \Delta_y \setminus \text{Int}(\mathcal{K})$, and after passing to subsequence, 
 \[
 l^{-1}p_l\to p_\infty\in    \Delta_y \setminus \text{Int}(\mathcal{K}),
 \]
 by the closedness of $ \Delta_y \setminus \text{Int}(\mathcal{K}).$
Upon taking the $C^0$-limit for the convex functions 
\[
 \frac{1}{l}  \log |\theta_\alpha^l | (x)-   \frac{1}{l}  \log |\theta_\alpha^l | (y)\to  c(x,p';y),
\]
we deduce that $p_\infty$ is a subgradient for $c(x,p'; y)$ at $x=y$. 
Hence on the $n$-dimensional open face $\Delta_J^0\subset Sk(X)$, we have
\[
c(x, p'; y)- c(y', p'; y) \geq \langle p_\infty, x-y'\rangle.
\]

By the definition of $c$-transform, for any $x\in \Delta_J^0\subset Sk(X)$, 
\[
\phi(x) \geq c(x,p'; y)- \phi^c(p') \geq c(y', p'; y)+  \langle p_\infty, x-y'\rangle - \phi^c(p').
\]
Using Step 2, 
\[
\phi(x) \geq c(y', p'; y)+  \langle p_\infty, x-y'\rangle - \phi^c(p') =  \phi(y') +  \langle p_\infty, x-y'\rangle .
\]
The upshot is that $p_\infty\in \nabla \phi(y')\notin \text{Int}(\mathcal{K})$, which contradicts Step 1. This concludes the proof for the claim. Thus we can regard $p'$ as an element in $\mathcal{K}\subset \Delta_y^0$.

\item  Consequently, by Prop. \ref{prop:limitingtropicalthetafromp2}, we must have
\[
c(x,p'; y)= \langle p', x-y\rangle, \quad \forall x\in U_y,
\]
and the maximal psh extension $\tilde{c}(x, p'; y)$ satisfies
\[
\tilde{c}(x,p'; y)= \langle p', r_\mathcal{X}(x)  -y\rangle,\quad \forall x\in r_{\mathcal{X}}^{-1}(U_y)\subset X_K^{an}.
\]

\item  By the definition of the $c$-transform,
\[
\phi(x)\geq c(x,p'; y)- \phi^c(p'),\quad \forall x\in Sk(X).
\]
By the domination property of the  potential $\phi$, 
\[
\phi(x) \geq \tilde{c}(x,p'; y) - \phi^c(p'),\quad \forall x\in X_K^{an}.
\]
For any $x\in r_{\mathcal{X}}^{-1}(U_y)\subset X_K^{an}$, using Step 5,  this implies
\[
\phi(x) \geq c( r_\mathcal{X}(x) , p';y  )- \phi^c(p').
\]
In particular, if $r_{\mathcal{X}}(x)=y'$, then using Step 2,
\[
\phi(x) \geq c( y', p';y  )- \phi^c(p')= \phi(y').
\]
Since this holds for any $y'\in U_y$, we conclude that $\phi \geq \phi\circ r_{\mathcal{X}}$ on $r_{\mathcal{X}}^{-1}(U_y)\subset X_K^{an}.$ The reverse inequality $\phi\leq \phi\circ r_{\mathcal{X}}$ is automatic for any semipositive potentials, hence $\phi= \phi\circ r_{\mathcal{X}}$ on $r_{\mathcal{X}}^{-1}(U_y)\subset X_K^{an}.$
This concludes the proof for the factorisation criterion.

	\end{enumerate}

	\end{proof}

	\subsection{Relative volume and Monge-Amp\`ere energy}\label{sec:relativevolume}

	We shall derive a formula for the Monge-Amp\`ere energy via the relative volume  (\cf Section \ref{sect:relativevolume}), similar to \cite[Section 5.4]{Litheta}.
We begin with some formula for Fubini-Study norms. Let $y\in Sk(X)$ be any given sufficiently irrational point.

	\begin{lem}\label{lem:FSdiagonal}
		Suppose  $\phi\in \text{CPSH}(X_K^{an},\mathcal{L})$ satisfies  the domination property (\ref{eqn:dominationproperty}).
		For $l$ large enough so that the $K$-vector space $H^0(X_K,lL)$ generates the line bundle $L^{\otimes l}\to X_K$, we define the Fubini-Study approximation
		\begin{equation}\label{eqn:phicl}
			\phi^c_l(\alpha)= \max_{x\in Sk(X)} (  \frac{1}{l} \log |\theta_\alpha^l|(x)-  \frac{1}{l} \log |\theta_\alpha^l|(y)-   \phi(x)),\quad \forall \alpha.
		\end{equation}
		Then the NA norms $\norm{\cdot}_{l\phi}$ on  $H^0(X_K,lL)$   is given by (See section \ref{sect:relativevolume})
		\[
		\norm{ \sum_{\alpha } a_\alpha \theta_\alpha^l }_{l\phi}= \max_\alpha |a_\alpha | e^{l \phi^c_l(\alpha) + \log |\theta_\alpha^l|(y) },\quad \forall a_\alpha \in K.
		\]
	\end{lem}

	\begin{proof}
		Let $s=\sum a_\alpha \theta_\alpha^l\in H^0(X_K,lL)$, where $a_\alpha\in K$. By definition
		\[
		\norm{s}_{l\phi}= \sup_{X_K^{an}} |s| e^{-l\phi},
		\]
		where $|s|$ is defined by the model line bundle $\mathcal{L}\to \mathcal{X}$. The valuative independence condition implies that
		\[
		|s|= \max_{\alpha } |a_\alpha| |\theta_\alpha^l| \quad \text{on $Sk(X)$},
		\]
		so by the definition of $\phi^c_l(\alpha)$, we have
		\[
		\norm{s}_{l\phi} \geq \max_{Sk(X) }   |s| e^{-l\phi} =\max_{x\in Sk(X)} \max_\alpha |a_\alpha| |\theta_\alpha^l| e^{-l\phi} =\max_{\alpha } |a_\alpha| e^{ l\phi^c_l(\alpha)+ \log |\theta_\alpha^l|(y)}.
		\]

		We now show the reverse inequality. We define the Fubini-Study potential on $X_K^{an}$, 
		\[
		\phi_l=\max_{\alpha}  l^{-1} \log |\theta_\alpha^l|   - l^{-1} \log |\theta_\alpha^l|(y)  - \phi^c_l(\alpha),
		\]		
		so $\phi_l\leq \phi$ on $Sk(X)$, whence $\phi_l\leq \phi$ on $X_K^{an}$ by the domination property of $\phi$. Thus
		\[
		\norm{ \theta_\alpha^l}_{l\phi} =  \sup_{X_K^{an}} |\theta_\alpha^l|  e^{-l\phi}
		\leq \sup_{X_K^{an}} |\theta_\alpha^l|  e^{-l\phi_l}\leq e^{l\phi_l^c(\alpha)  + \log |\theta_\alpha^l|(y) },
		\]
		so the ultrametric inequality implies
		\[
		\norm{s}_{l\phi} \leq \max_\alpha |a_\alpha| \norm{\theta_\alpha^l}_{l\phi} =\max_ \alpha |a_\alpha | e^{l\phi_l^c(\alpha) + \log |\theta_\alpha^l|(y) },
		\]
		completing the proof of the reverse inequality. 
	\end{proof}

	\begin{Notation}
		Let $y$ be any given sufficiently irrational point in $Sk(X)$.
		The convex set $\Delta_y^0\subset \R^n$ has a canonical Lebesgue measure determined by its integral structure.
	We define the measure $\tilde{\nu}$ on $B_y$ as the pushforward of this Lebesgue measure via the embedding $\Delta_y^0\to B_y$ in Prop. \ref{prop:homeotoimage}, which in particular puts zero measure on the complement of the image of this embedding. We write $\tilde{\nu}= \frac{(L^n)}{n!} \nu$. By Cor. \ref{cor:semigroup}, the Lebesgue volume of the Okounkov body is $\frac{(L^n)}{n!}$, hence $\nu$ is a probability measure.
	\end{Notation}

	We now present the formula for the Monge-Amp\`ere energy, following the approach of Boucksom-Eriksson \cite{Boucksomnew1}, which is inspired by the Chebyshev transform of Witt-Nystr\"om \cite{WittNystrom}.

	\begin{prop}\label{prop:MAenergy}
	Suppose $\phi\in \text{CPSH}(X_K^{an},\mathcal{L})$ satisfies the domination property (\ref{eqn:dominationproperty}).
		Then up to an additive normalisation constant, the Monge-Amp\`ere energy is
		\begin{equation}
		E(\phi)= -(L^n) \int_{B_y}  \phi^c(p) d\nu.
		\end{equation}
	
	\end{prop}

		\begin{proof}
			We divide the proof into a few steps.

			\begin{enumerate}
				\item  	Take any other potential $\psi\in \text{CPSH}(X_K^{an},\mathcal{L})$ with the domination property. Then by Lemma \ref{lem:FSdiagonal}, the NA norms $\norm{\cdot}_{l\phi}$ versus $\norm{\cdot}_{l\psi}$ have a simultaneous orthogonal basis $\theta_\alpha^l$, namely for any $s=\sum_\alpha a_\alpha \theta_\alpha^l $ with $a_\alpha\in K$, 
				\[
				\norm{s }_{l\phi}= \max_\alpha |a_\alpha| e^{l \phi^c_l(\alpha) + \log |\theta_\alpha^l|(y)},\quad 
				\norm{s }_{l\psi}= \max_\alpha |a_\alpha | e^{l \psi^c_l(\alpha) + \log |\theta_\alpha^l|(y)}.
				\]
				Thus  the relative volume between the two NA norms  is (see Section \ref{sect:relativevolume})
				\begin{equation}\label{eqn:relativevolume1}
					vol(  \norm{\cdot}_{l\phi}, \norm{\cdot}_{l\psi}  )= \log \prod_{  \alpha } e^{-l \phi^c_l(\alpha)+l \psi^c_l(\alpha)} = l\sum_{ \alpha }  ( - \phi^c_l(\alpha)+\psi^c_l(\alpha) ).
				\end{equation}
				By Section \ref{sect:relativevolume} and (\ref{eqn:relativevolume1}), 
				\[
				\begin{split}
					E(\phi)-E(\psi)=& \lim_{l\to+\infty} \frac{n!}{  l^{n+1}  } 	vol(  \norm{\cdot}_{l\phi}, \norm{\cdot}_{l\psi}  )
					\\
					= & n!  \lim_{l\to+\infty} l^{-n} \sum_{  \alpha }  (-\phi^c_l(\alpha )+ \psi^c_l(\alpha)).
				\end{split}
				\]

				\item Let $\mathcal{K}\subset \Delta_y^0$ be any fixed closed convex polyhedral subset. We can partition $\{ \theta_\alpha^l \}$ according to whether $p(\alpha)=(\nabla \log |\theta_\alpha^l|)(y) \in \Gamma_{y,l}\cap l \mathcal{K}$.

				For all $p(\alpha)\in l \mathcal{K}\cap \Gamma_{y,l}$, by Prop. \ref{prop:limitingtropicalthetafromp2} item 4, we have the uniform convergence as $l\to +\infty$,
					\[
				\sup_{ x\in Sk(X) , l^{-1}p(\alpha) \in \mathcal{K} } |    l^{-1}(\log |\theta_\alpha^l|(x)- \log |\theta_\alpha|(y) )- c(x,  l^{-1} p(\alpha);y) | \leq \epsilon_l\to 0.
				\]
				Thus we can compare the $c$-transforms,
				\[
				|\phi_l^c(\alpha) - \phi^c( l^{-1} p(\alpha) ) | \leq \epsilon_l.
				\]
				Consequently, 
				\[
				\sum_{ \alpha: p(\alpha)\in l \mathcal{K} }  ( - \phi^c_l(\alpha)+\psi^c_l(\alpha) )= \sum_{ \alpha: p(\alpha)\in l\mathcal{K} \cap \Gamma_{y,l} }  ( - \phi^c( l^{-1} p(\alpha) )+\psi^c ( l^{-1}p(\alpha) ) + O(\epsilon_l)).
				\]
				The function $\phi^c(p)$ is Lipschitz continuous on $\mathcal{K}$ by Prop. \ref{prop:limitingtropicalthetafromp} item 3. Moreover, for large enough $l$, then $l \mathcal{K}\cap \Z^n= l \mathcal{K}\cap \Gamma_{y,l}$ by Cor. \ref{cor:semigroup} item 1.

				Thus the Riemann sum converges as $l\to +\infty$,
				\[
					l^{-n} \sum_{  p(\alpha)\in l\mathcal{K} }  \phi^c( l^{-1}p(\alpha))\to \int_\mathcal{K}  \phi^c(p)   d\tilde{\nu} = \frac{   ( L^n)}{n!}  \int_{\mathcal{K}}   \phi^c(p)   d\nu.
				\]
				Since the total number of integer points in $l\mathcal{K}$ is $O(l^n)$, the $O(\epsilon_l)$ term can be neglected, so 
				\[
				\lim_{l\to \infty} l^{-n} \sum_{ \alpha: p(\alpha)\in l \mathcal{K} }  ( - \phi^c_l(\alpha)+\psi^c_l(\alpha) )= - \frac{   ( L^n)}{n!}  \int_{\mathcal{K}}   (\phi^c(p)-\psi^c(p) )   d\nu.
				\]

				\item  The total number of integer points in $l\mathcal{K}$ is $l^n \text{Vol}(\mathcal{K}) (1+o(1))$, while the total number of $\theta_\alpha^l$ is 
				\[
				\dim_K H^0(X_K, lL) = \frac{ (L^n)}{n! } l^n(1+O(l^{-1}) ) = l^n \text{Vol}(\Delta_y) (1+o(1)).
				\]
				Since $\phi_l^c(\alpha), \psi_l^c(\alpha)$ are uniformly bounded independent of $l, \alpha$, we can estimate
				\[
				|  l^{-n} \sum_{  p(\alpha)\notin l \mathcal{K} }  (-\phi^c_l(\alpha )+ \psi^c_l(\alpha)) | \leq C l^{-n} | \{    \alpha: p(\alpha) \notin l\mathcal{K}       \}| \leq C\text{Vol}(\Delta_y\setminus \mathcal{K}).
				\]

				\item 
				
				Combining Step 2,3, and taking an exhaustion of $\Delta_y^0$ by an increasing sequence of closed convex polyhedral subsets $\mathcal{K}$, we deduce
				\[
					\lim_{l\to \infty} l^{-n} \sum_{ \alpha}  ( - \phi^c_l(\alpha)+\psi^c_l(\alpha) )= - \frac{   ( L^n)}{n!}  \int_{B_y}   (\phi^c(p)-\psi^c(p) )   d\nu.
				\]
				By Step 1,
				\[
					E(\phi)-E(\psi)
				=  n!  \lim_{l\to+\infty} l^{-n} \sum_{  \alpha }  (-\phi^c_l(\alpha )+ \psi^c_l(\alpha))= -(L^n) \int_{B_y}   (\phi^c(p)-\psi^c(p) )   d\nu.
				\]
				The Monge-Amp\`ere energy formula follows.
			\end{enumerate}
	\end{proof}

	\section{NA MA equation and optimal transport}\label{sec:NAMAoptimaltransport}

	We continue with the setup of Section \ref{sec:valuativeindependence}, and assume the existence of the valuative independent basis $\{ \theta_\alpha^l \}\subset H^0(X_K, lL)$ for any $l\geq 0$. Throughout this Section, we suppose that $\varphi\in \text{CPSH}(X_K^{an},\mathcal{L})$ solves the NA MA equation
	\begin{equation}
			\text{MA}(\varphi )=(L^n) \mu,
	\end{equation}
	for some probability measure $\mu$ supported on $Sk(X)\subset X_K^{an}$, such that $\mu$ 
	is \emph{absolutely continuous} with respect to 
	 the Lebesgue measure $\mu_0$ on $Sk(X)$. In the special case $\mu=\mu_0$, then $\varphi=\phi_0$ is the NA CY potential up to constant.
Our goal is to prove a version of the weak comparison property for $\varphi$, as well as deriving an optimal transport interpretation for the NA MA equation.

	\subsection{Weak comparison property}\label{sec:weakcomparison}

A key ingredient in proving the weak comparison property is that certain bad events can only happen on null  $\mu$-measure sets. The proof  subtlely uses the \emph{orthogonality property} for psh envelops  (\cf Section \ref{sec:envelop}).

\begin{Notation}
Let $y\in Sk(X)$ be fixed. The points $x\in Sk(X)$ and $p\in B_y$ are called \emph{conjugate} points for the function $\varphi$, if $\varphi(x)= c(x,p;y)-\varphi^c(p)$.
We define the subsets 
\[
\begin{cases}
E_0= \{     x\in Sk(X):    \text{$x$ is sufficiently irrational}     \}= Sk(X)\setminus \bigcup_{l\geq 1} \mathcal{S}_l  ,
\\
E_1= \{       x\in E_0:     \nabla   \varphi(x) \not\subset \text{Int}(\Delta_x)   \},
\\
E_2= \{      x\in E_0\setminus E_1:    \text{$\exists x'\neq x\in Sk(X), p\in B_y$, with $x,x'$ both conjugate to $p$}      \}.
\end{cases}
\]

\end{Notation}

	\begin{lem}\label{lem:nullmeasure}
The sets $E_1, E_2$ have zero $\mu$-measure.
	\end{lem}

	\begin{proof}
		We divide the proof into a few steps.
		\begin{enumerate}
			\item  Since $\varphi$ is a convex function on each open $n$-dimensional face $\Delta_J^0\subset Sk(X)$, the set
			\[
			E_0'= \{    x\in Sk(X): \text{the subgradient set $\nabla \varphi(x)$ is a singleton}     \}
			\]
			has full Lebesgue measure, hence full $\mu$-measure by absolute continuity. Let $E_1'=E_1\cap E_0'$ and $E_2'=E_2\cap E_0'$, then it suffices to prove that $\mu(E_1')=\mu( E_2')=0$.

			\item Since $\varphi\in  \text{CPSH}(X_K^{an}, \mathcal{L})$ and $\mu= \text{MA}(\varphi)$, Cor. \ref{cor:orthogonality} implies the following result on psh envelops.  Let $f:X_K^{an}\to \R$ be any given model function,  then for small $\tau>0$, 
			\begin{equation}
			\mu( \{      P(\varphi+\tau f) <  \varphi+\tau  f        \}   ) \leq C\tau,
			\end{equation}
			for some constant $C$ depending on $f, \varphi$, but independent of $\tau$.

			Our argument will be based on testing various model functions $f$. These $f$ come from 
 vertical $\Q$-divisors on all the models of the degeneration family, so forms a  countable set.

			\item For any $x\in E_1'$, then $x\in E_0\cap E_0'$, so $x$ is sufficiently irrational, and there is a unique gradient $p=\nabla \varphi(x)$. Since $\varphi\in  \text{CPSH}(X_K^{an}, \mathcal{L})$, we know $p$ belongs to the gradient convex hull $\mathcal{C}_x$ (\cf Section \ref{sec:gradientconvexhullOkounkovbody}). Since $p$ is the unique gradient, it belongs to the Okounkov body $\Delta_y\subset \mathcal{C}_y$ by Lemma \ref{lem:gradientconvexhullOkounkovbody} item 2. Since $x\in E_1'\subset E_1$, the gradient $p$ is not in the interior $\Delta_y^0$, hence $p\in \partial \Delta_y$.

			Given any model function $f$,  let $E_1^f$ be the subset consisting of the points $x\in E_1'$, such that $f$ is affine linear in some open neighbourhood of $x\in Sk(X)$, and $\nabla \varphi(x) + \tau \nabla f (x)\notin \Delta_x$ at the point $x$ for any $\tau>0$. This means that at the boundary point $p\in \partial \Delta_x$, the vector $\nabla f(x)$ points outside of the convex Okounkov body $\Delta_x$.

		We notice that at any $x\in E_1'$, one can find such a model function $f$, so 
	$
		E_1'= \bigcup_f  E_1^f.
	$
			
			\item  We claim that at any $x_0\in E_1^f$,
			\begin{equation}
			 P(\varphi+\tau f) (x_0) < ( \varphi+\tau  f)(x_0)  ,\quad \forall \tau>0.
			\end{equation}
	Suppose not, then $P(\varphi+\tau f)(x_0)= \varphi(x_0) +\tau  f(x_0)$. On $ Sk(X)$, we have
			\[
			 P(\varphi+\tau f) \leq   \varphi+\tau  f .
			\]
	The function $P(\varphi+\tau f) \in  \text{CPSH}(X_K^{an},\mathcal{L})$  restricts to a function in $\mathcal{P}_c$, which is in particular convex on the $n$-dimensional faces of $Sk(X)$. The above inequality shows that any subgradient $p'\in (\nabla  P(\varphi+\tau f)  )(x_0)$ is also a subgradient of the convex function $\varphi+\tau  f$ at $x_0$, which can only be 
	\[
	p'= \nabla \varphi(x_0) +\tau \nabla f(x_0).
	\]
	In particular, this subgradient $p'$ is unique, while $P(\varphi+\tau f)\in \mathcal{P}_c$, so $p'\in \Delta_{x_0}$ by Lemma \ref{lem:gradientconvexhullOkounkovbody}. But by the definition of $E_1^f$, 
	\[
	 \nabla \varphi(x_0) +\tau (\nabla f)(x_0) \notin \Delta_{x_0}, \quad \forall \tau>0.
	\]
			This contradiction proves the claim.

			\item By Step 4, the set $E_1^f\subset \{    P(\varphi+\tau f)  <  \varphi+\tau  f       \}$ for any $\tau>0$. Hence
			\begin{equation}
			\mu(E_1^f) \leq \mu(    \{    P(\varphi+\tau f) <  \varphi+\tau  f        \}    ) \leq C\tau.
			\end{equation}
			The constant $C$ depends on $\varphi, f$, but not on $\tau$. Taking $\tau\to 0$, we obtain $\mu(E_1^f)=0$. Since $E_1'= \cup_f E_1^f$ is a countable union, we deduce $\mu(E_1')=0$.

		\item  For any $x\in E_2'$, then $x\in E_0\cap E_0'$, so $x$ is sufficiently irrational, and there is a unique gradient $p=\nabla \varphi(x)$. Since $x\in E_2'\subset E_2$ is disjoint from $E_1$, we deduce that $p\in \Delta_x^0$ is in the interior of the Okounkov body. 
		Moreover, by the definition of $E_2$, there is some $x'\neq x\in Sk(X)$, such that $x,x'$ are both conjugate to some $p'\in B_y$, namely
		\[
		\varphi(x)+ \varphi^c(p')= c(x,p';y),\quad \varphi(x')+ \varphi^c(p')= c(x',p';y).
		\]

		Given any model function $f$, we let $E_2^f$ to be  the subset consisting of the points $x\in E_2'$, such that
		\begin{itemize}
			\item  The function $f=0$ on some open neighbourhood of $x\in Sk(X)$,
			\item   There is some $x'\neq x\in Sk(X)$, and $p'\in B_y$, such that $x,x'$ are both conjugate to the same $p'\in B_y$, and $f(x')<0$.
		\end{itemize}
		We notice that at any $x\in E_2'$, we can always find such a model function $f$, so $E_2'= \bigcup_f E_2^f$.

		\item    We claim that at any $x_0\in E_2^f$,
		\[
		P(\varphi+\tau f) (x_0) < ( \varphi+\tau  f)(x_0)  ,\quad \forall \tau>0.
		\]
		Suppose not, then $P(\varphi+\tau f)(x_0)= \varphi(x_0) +\tau  f(x_0)$. On $ Sk(X)$, we have
		\[
		P(\varphi+\tau f) \leq  ( \varphi+\tau  f).
		\]
		We denote $p_0=\nabla \varphi(x_0)$. 
		By the same argument as in Step 4, any subgradient $p_0' \in ( \nabla P(\varphi+\tau f) )(x_0)$
		is also a subgradient of the convex function $\varphi+\tau f$ at $x_0$. 
		But $f=0$ in an open neighbourhood of $x_0\in Sk(X)$, so
	\begin{equation}
p_0'= (	\nabla ( \varphi+\tau f  ) )(x_0)= \nabla \varphi (x_0)=p_0.
	\end{equation}
	The first part of Step 6 implies $p_0'= p_0\in \Delta_{x_0}^0$.

	By Prop. \ref{prop:homeotoimage}, $\Delta_{x_0}^0$ \emph{injects} into $B_{x_0}$, which determines a \emph{unique limiting tropical theta function} $c(\cdot, p_0; x_0)$ in $B_{x_0}$ whose gradient at $x_0$ is $p_0$. On the other hand, $x_0$ is conjugate to $p'\in B_y$, so 
	\[
	\varphi \geq c(\cdot, p'; y)- \varphi^c(p'), \quad 	\varphi(x_0)= c(x_0, p'; y)- \varphi^c(p'),
	\]
	hence the gradient of $c(\cdot, p'; y)$ at the point $x_0$ must agree with  $\nabla \varphi(x_0)=p_0$. Now by Lemma \ref{lem:changey} which compares $B_y$ with $B_{x_0}$, the function
	\[
	c(\cdot, p'; y)- c(x_0, p'; y) \in B_{x_0.}
	\]
The gradient of this function at $x_0$ is still $p_0$. Hence by uniqueness,
	\begin{equation}
	c(\cdot, p_0; x_0)= c(\cdot, p'; y)- c(x_0, p'; y).
	\end{equation}

Since $P(\varphi+\tau f)$ lies in $\mathcal{P}_c$, by Lemma \ref{lem:ctransformmax} there must be some $p''\in B_y$ conjugate to $x_0$ for the function $P(\varphi+\tau f)$. By running the same argument for $P(\varphi+\tau f)$ as what we just did for $\varphi$, we deduce that
\begin{equation}
		c(\cdot, p_0' ; x_0)= c(\cdot, p''; y)- c(x_0, p'' ; y).
\end{equation}
	But $p_0=p_0' \in \Delta_{x_0}^0$, so comparing the above gives
	\[
	c(\cdot, p'; y)= c(\cdot, p''; y)+ \text{const}.
	\] 
	Since $c(y, p'; y)= c(y,p'';y)=0$, this shows that $p'=p''\in B_y$.

	Since $x_0$ and $p'=p''\in B_y$ are conjugate points for $g=P(\varphi+\tau f)$, by definition
	\[
	g (x_0) +  	g^c(p') = c(x_0 ,p';  y ),
	\]
hence for any $x\in Sk(X)$,
\begin{equation}
\begin{split}
(\varphi+\tau f)(x) \geq g(x) \geq &c(x, p' ; y)- g^c(p')
\\
= &  c(x,p'; y)- c(x_0, p';y)+ g(x_0)
\\
= & c(x,p'; y)- c(x_0, p';y)+ \varphi(x_0).
\end{split}
\end{equation}
where the last step used the contradiction hypothesis $P(\varphi+\tau f)(x_0)= (\varphi+\tau f)(x_0)= \varphi(x_0)$.

	Now by the definition of $E_2^f$, there is another point $x_0'$ with $f(x_0')<0$, which is also conjugate to the same $p'$ for the function $\varphi$, hence
	\[
	\varphi(x_0') + \varphi^c(p') = c(x_0', p'; y),\quad   	\varphi(x_0) + \varphi^c(p') = c(x_0, p'; y)  . 
	\]
Therefore
	\[
	\begin{split}
		 (\varphi+\tau f)(x_0') &\geq c(x_0',p'; y)- c(x_0, p';y)+ \varphi(x_0)
			\\
			& = \varphi(x_0') + \varphi^c(p') -c(x_0, p'; y) + \varphi(x_0) 
			\\
			& =  \varphi(x_0') ,
	\end{split}
	\]
	hence $f(x_0')\geq 0$, contradicting the definition of $E_2^f$. 
 This proves the claim.

	\item  By Step 7,  the set $E_1^f\subset \{   P(\varphi+\tau f)  < \varphi+\tau  f        \}$ for any $\tau>0$. Hence
	\begin{equation}
		\mu(E_2^f) \leq \mu(    \{     P(\varphi+\tau f)  <  \varphi+\tau  f        \}    ) \leq C\tau.
	\end{equation}
	The constant $C$ depends on $\varphi, f$, but not on $\tau$. Taking $\tau\to 0$, we obtain $\mu(E_2^f)=0$. Since $E_2'= \cup_f E_2^f$ is a countable union, we deduce $\mu(E_2')=0$.

		We have proved $\mu(E_1')= \mu(E_2')=0$, hence $\mu(E_1)=\mu(E_2)=0$ by Step 1.
		
		\end{enumerate}

	\end{proof}

\begin{thm}
(Weak comparison property for absolutely continuous NA MA 
measures)
\\
Suppose that $\varphi\in \text{CPSH}(X_K^{an},\mathcal{L})$ solves the NA MA equation
$
	\text{MA}(\varphi )=(L^n) \mu,
$
for some probability measure $\mu$ supported on $Sk(X)\subset X_K^{an}$, such that $\mu$ 
is \emph{absolutely continuous} with respect to 
the Lebesgue measure $\mu_0$ on $Sk(X)$. Then there is some open subset $U$ contained in the $n$-dimensional open faces of $Sk(X)$, with full measure $\mu(U)=1$, such that $\varphi= \varphi\circ r_{\mathcal{X}}$ holds for the retraction map $r_\mathcal{X}: X_K^{an}\to \Delta_\mathcal{X}$ on the set $r_\mathcal{X}^{-1}(U)$.

In particular, for the NA CY potential $\phi_0$ solving $\text{MA}(\phi_0)= (L^n)\mu_0$, then $U\subset Sk(X)$ is an open dense subset with full Lebesgue measure.
\end{thm}

	\begin{proof}
	Since $\mu$ is absolutely continuous with respect to the Lebesgue measure, the set $E_0$ has full measure $\mu(E_0)=1$. Since $\mu(E_1)=0$ by Lemma \ref{lem:nullmeasure}, we see $\mu(E_0\setminus E_1)=1$. For every $x\in E_0\setminus E_1$, the point $x$ is sufficiently irrational, and $\nabla \varphi(x) \subset \Delta_x^0$.

	Since $\varphi\in \text{CPSH}(X_K^{an},\mathcal{L})$ has its NA MA measure supported on $Sk(X)$, Lemma \ref{lem:dominationproperty} implies that $\varphi$ satisfies the domination property (\ref{eqn:dominationproperty}). We can now apply the factorisation criterion Prop. \ref{prop:factorisation}. At any $x\in E_0\setminus E_1$, there is some open neighbourhood $U_x\subset Sk(X)$, such that $\varphi= \varphi\circ r_{\mathcal{X}}$ holds over $U_x$. Taking the union $U:= \bigcup_{x\in E_0\setminus E_1} U_x$, then $U\subset Sk(X)$ is an open subset with full measure $\mu(U)=1$, and $\varphi= \varphi\circ r_{\mathcal{X}}$ holds over $U$.

	Finally, if  $\mu$ puts nontrivial measure on every open subset of $Sk(X)$, then the full measure open subset $U$ must be dense in $Sk(X)$. This applies to the NA CY potential.
	\end{proof}

	On the open $n$-dimensional faces of $Sk(X)$, there is an integral affine structure, so we can make sense of the real Monge-Amp\`ere measure $\text{MA}_\R(\varphi)$.

	\begin{cor}
	On the open set $U$, the potential $\varphi$ solves the real Monge-Amp\`ere equation
$
	\text{MA}_\R(\varphi) =   \frac{1}{n!}\mu. 
$
	\end{cor}
	
	\begin{proof}
	Since $\varphi=\varphi\circ r_{\mathcal{X}}$ over $U$, by the result of Vilsmeier \cite{Vilsmeier},
	\[
	\text{MA}_\R(\varphi) |_U=  \frac{1}{n!}	r_{\mathcal{X}* } \text{MA}( \varphi )= \frac{1}{n!}\mu. 
	\]
	\end{proof}

	\begin{rmk}
If one wishes to directly study the real Monge-Amp\`ere equation on $Sk(X)$, one would face the following basic difficulty. The structure of $Sk(X)$ as a simplicial complex depends on the choice of the model $\mathcal{X}$, and $Sk(X)$ has no canonical affine structure independent of models. 
If the transition functions between local charts are not affine linear, then the notion of convex functions or the real Monge-Amp\`ere equation is not invariantly defined. The affine structure (here the subset $U$) is determined by the solution $\varphi$, rather than given a priori.

Our arguments above does not yield any topological information on $U$. It is an interesting question for the NA CY potential $\phi_0$, whether there is some Hausdorff codimension two subset of $Sk(X)$, such that its complement admits some affine structure, wherein $\phi_0$ solves the real MA equation.
	\end{rmk}




	\subsection{Optimal transport formulation}

The NA MA equation can be formulated as the Kontorovich dual of an optimal transport problem.

\begin{thm} \label{thm:optimaltransport}
Let $\mu$ be any probability measure supported on $Sk(X)\subset X_K^{an}$. 
The potential $\varphi\in \text{CPSH}(X_K^{an},\mathcal{L})$ solves the NA MA equation
$
	\text{MA}(\varphi )=(L^n) \mu,
$
if and only if its restriction to $Sk(X)$ is the unique (up to constant) minimiser of the functional $\mathcal{F}_\mu: \mathcal{P}_c\to \R$,
\[
\mathcal{F}_\mu(\phi)=  \int_{B_y}  \phi^c(p)   d\nu+  \int_{Sk(X)} \phi d\mu.
\]
\end{thm}

\begin{proof}
We divide the proof into a few steps.

\begin{enumerate}
	\item  By Prop. \ref{prop:maxextension2},  any $\phi\in \text{CPSH}(X_K^{an},\mathcal{L})$ restricts on $Sk(X)$ to a function in $\mathcal{P}_c$. By Prop. \ref{prop:maxextension}, each $\phi\in \mathcal{P}_c$ has a unique maximal psh extension to $\text{CPSH}(X_K^{an},\mathcal{L})$ satifying the domination property (\ref{eqn:dominationproperty}). We can thus identify $\mathcal{P}_c$ with the potentials in $\text{CPSH}(X_K^{an},\mathcal{L})$ satisfying the domination property.

  For any $\phi\in \text{CPSH}(X_K^{an},\mathcal{L})$ satisfying the domination property, 
	the functional  $F_\mu: \text{CPSH}(X_K^{an},\mathcal{L})\to \R$ (\cf Section \ref{sec:NAMAeqn}) is
	\[
F_{\mu}(\phi)= -  \frac{1}{(L^n)} E(\phi)+ \int_{X_K^{an}} \phi d\mu =  \int_{B_y}  \phi^c(p)   d\nu + \int_{Sk(X)} \phi d\mu =\mathcal{F}_\mu(\phi)
	\]
by Prop. \ref{prop:MAenergy}.

	\item 
	Since the NA MA measure $\text{MA}(\varphi)$ is supported on $Sk(X)$, it satisfies the domination property (\ref{eqn:dominationproperty}) by Lemma \ref{lem:dominationproperty}. As discussed in Section \ref{sec:NAMAeqn}, the solution $\varphi$ to the NA MA equation minimises  $F_\mu$ among $ \text{CPSH}(X_K^{an},\mathcal{L})$, and in particular is the minimiser among those potentials satisfying the domination property. Hence $\varphi\in \mathcal{P}_c$ minimises $\mathcal{F}_\mu$ among $\mathcal{P}_c$.

	\item Conversely, suppose $\varphi'\in \mathcal{P}_c$ minimises $\mathcal{F}_\mu$. We take the maximal psh extension $\varphi'\in  \text{CPSH}(X_K^{an},\mathcal{L})$, so 
	\[
	F_\mu(\varphi')= \mathcal{F}_\mu( \varphi ')= \min_{\mathcal{P}_c} \mathcal{F}_\mu = \mathcal{F}_\mu(\varphi)= F_\mu(\varphi).
	\]
	Thus $\varphi, \varphi'$ are both minimisers of $F_\mu: \text{CPSH}(X_K^{an},\mathcal{L})\to \R$, hence both solve the same NA MA equation
	\[
	\text{MA}(\varphi)= \text{MA}(\varphi')= (L^n) \mu.
	\]
	But the solution to the NA MA equation is unique up to constant, hence $\varphi= \varphi'+c$. This proves the uniqueness statement.
\end{enumerate}

\end{proof}

	\subsection{A global real Monge-Amp\`ere type equation}

For motivation, recall that for a convex function $u:\Omega\to \R$ on some open domain $\Omega\subset \R^n$, 
the \emph{Alexandrov weak solution to the real Monge-Amp\`ere equation} $\det(D^2 u)=1$ is defined by 
\[
|\nabla u(E) |= |E|,\quad  \forall E\subset \Omega,
\]
where $\nabla u(E)= \{     p\in \R^n:  p\in \nabla u(x) \text{ for some }  x\in E    \}.
$
The notion of conjugate sets provides a global analogue for this weak formulation, for the solution $\varphi$ to the NA MA equation.

	Let $y$ be any fixed sufficiently irrational point on $Sk(X)$. 
Let $\varphi\in \mathcal{P}_c$ be the solution to the NA MA equation $\text{MA}(\varphi)=(L^n)\mu$, where $\mu$ is a probability measure supported on $Sk(X)$, which is absolutely continuous with respect to the Lebesgue measure. Recall we have  a probability measures $\nu$ on  $B_y$, coming from the Lebesgue measure on $\Delta_y^0$ (\cf Section \ref{sec:relativevolume}).

\begin{Notation}
For any Lebesgue measurable subsets $E\subset Sk(X)$ (resp. $F\subset B_y$), we define the \emph{conjugate sets}
\[
\begin{cases}
	\bar{\nabla}\varphi (E):= \{    p\in B_y:  \varphi(x)+ \varphi^c(p)=c(x,p;y) \text{ for some }  x\in E            \},
	\\
		\bar{\nabla}\varphi (F):= \{    x\in Sk(X):  \varphi(x)+ \varphi^c(p)=c(x,p;y) \text{ for some }  p\in F           \}.
\end{cases}
\]

\end{Notation}

\begin{lem}\label{lem:globalrealMA}
For any measurable $E\subset Sk(X)$ (resp. $F\subset B_y$), we have
\[
\nu( 	\bar{\nabla}\varphi (E  ) )\geq  \mu(E), \quad \mu(	\bar{\nabla}\varphi (F))\geq \nu(F).
\]

\end{lem}

	\begin{proof}
	By Thm. \ref{thm:optimaltransport}, the solution $\varphi$ to the NA MA equation is the minimiser for the functional
	\[
	\mathcal{F}_\mu(\phi)=  \int_{B_y}  \phi^c(p)   d\nu+  \int_{Sk(X)} \phi d\mu.
	\]
	The rest of the argument is similar to \cite[Prop. 2.16]{LiFano}.

		We consider the function $\varphi-\tau 1_E$ for $0<\tau \ll 1$, where $1_E$ denotes the characteristic function of $E$. By elementary properties of the $c$-transform,
		\[
		\varphi^c \leq (\varphi-\tau 1_E)^c\leq \varphi^c+\tau.
		\]
		Moreover $(\varphi-\tau 1_E)^c >\varphi^c$ at $p\in B_y$, only when 
		\[
		\sup_{x\in E}( c(x,p;y)-\varphi(x) )\geq \varphi^c(p)-\tau.
		\]
		Hence
		\[
		\begin{split}
			\tau \nu (\{ p: \sup_{x\in E} ( c(x,p;y)-\varphi(x)  ) \geq \varphi^c(p)-\tau  \} )
			\geq  \int_{B_y} ((\varphi-\tau 1_E)^c-\varphi^c)  d\nu.
		\end{split}
		\]

		Since $\varphi$ is a minimiser of $\mathcal{F}_\mu$, plugging in $(\varphi-\tau 1_E)^{cc}\in \mathcal{P}_c$ as a competitor, whose $c$-transform is $(\varphi-\tau 1_E)^c$ by Lemma \ref{lem:Pc}, 
		we get
		\[
		\int_{B_y} ((\varphi-\tau 1_E)^c-\varphi^c)d\nu \geq - \int_{Sk(X)} ((\varphi-\tau 1_E)^{cc}- \varphi)d\mu \geq - \int_{Sk(X)} (\varphi-\tau 1_E)- \varphi=\tau \mu(E).
		\]
		Here the first inequality follows by the minimisation property of $\varphi$, and the second inequality uses Lemma \ref{lem:Pc}. Combining the above and cancelling the $\tau $ factor, 
		\[
		\nu( \{ p: \sup_{x\in E} ( c(x,p;y) -\varphi(x)  ) \geq \varphi^c(p)-\tau   \} ) \geq \mu(E), \quad \forall 0<\tau \ll 1. 
		\]
		Taking the $\tau \to 0$ limit, and recalling $\varphi^c(p)\geq c(x,p;y)-\varphi(x)$, we get
		\[
		\nu( \{ p: \sup_{x\in E} ( c(x,p;y)-\varphi(x)  ) = \varphi^c(p)   \} )  \geq \mu(E).
		\]

		For closed subsets $E$, the sup can be replaced by max by the continuity of $\varphi$, so the LHS subset is the conjugate set  $\bar{\nabla}\varphi(E)$. For general Borel subsets $E$, we take a compact exhaustion of $E$. For any compact subset $K\subset E$, we observe
		\[
		\nu(  \bar{\nabla}\varphi(E) )\geq \nu( \bar{\nabla} \varphi(K))  \geq \mu(K).
		\]
		Since the measure $\mu$ is absolutely continuous with respect to the Lebesgue measure, we obtain $\nu(\bar{\nabla}\varphi(E)) \geq\mu(E) $ by taking the limit $K\uparrow E$.

		The other inequality $\mu(	\bar{\nabla}\varphi (F))\geq \nu(F)$ is entirely similar.
	\end{proof}

	\begin{thm}\label{thm:globalrealMA}
		(Global real Monge-Amp\`ere type equation)
	For any Lebesgue measurable $E\subset Sk(X)$ (resp. $F\subset B_y$), we have
	\[
	\nu( 	\bar{\nabla}\varphi (E  ) )= \mu(E), \quad \mu(	\bar{\nabla}\varphi (F))= \nu(F).
	\]
	\end{thm}

	\begin{proof}
	Let $E''= \bar{\nabla}\varphi ( \bar{\nabla}\varphi (E)    )\subset Sk(X)$. 
	By the definition of conjugate sets, clearly $E\subset E''$. Moreover, for any point $x\in E''\setminus E$, there is some $p\in B_y$, which is conjugate to both $x$ and some $x'\in E$. Thus recalling the sets $E_0, E_1, E_2$ from Section \ref{sec:weakcomparison},
	\[
	E'' \setminus E\subset  (Sk(X)\setminus E_0)\cup E_1\cup E_2.
	\]
	Since $\mu$ is absolutely continuous, $\mu(Sk(X)\setminus E_0)=0$. By Lemma \ref{lem:nullmeasure}, $\mu(E_1)=\mu(E_2)=0$, hence $\mu(E''\setminus E)=0$. 
	But
	by Lemma \ref{lem:globalrealMA}, 
	\[
	\nu( 	\bar{\nabla}\varphi (E  ) )\geq \mu(E)= \mu( E''  )\geq \nu(  	\bar{\nabla}\varphi (E  )  ).
	\]
	This forces the equality $\nu( 	\bar{\nabla}\varphi (E  ) )= \mu(E)$.

	By Lemma \ref{lem:globalrealMA}, we have $ \mu(	\bar{\nabla}\varphi (F))\geq  \nu(F)$. Let 
	\[
	F'= 	\bar{\nabla}\varphi (F) \setminus  ( (Sk(X)\setminus E_0)\cup E_1\cup E_2 ),
	\]
	then $\mu(F')= \mu(   	\bar{\nabla}\varphi (F)    )$, since the deleted subset has zero measure. Any $x\in F'$ has a unique conjugate point $p\in B_y$, since it avoids $(Sk(X)\setminus E_0)\cup E_1\cup E_2$. But $F'\subset \bar{\nabla}\varphi (F) $ means that $x$ must be conjugate to some point in $F$, hence $p\in F$. This shows $\bar{\nabla}\varphi (F')\subset F$. By Lemma \ref{lem:globalrealMA} again, 
	\[
 \mu(   	\bar{\nabla}\varphi (F)    )=	\mu(F') \leq \nu(  \bar{\nabla}\varphi (F')     )\leq \nu(F) \leq  \mu(   	\bar{\nabla}\varphi (F)    ).
	\]
	This forces the equality $ \mu(   	\bar{\nabla}\varphi (F)    )= \nu(F)$. 
	\end{proof}

	\begin{rmk}
	For different choices of sufficiently irrational $y\in Sk(X)$, we can canonically identify $B_y$ using \ref{lem:changey}. This identification is compatible with the notion of conjugate points, hence compatible with the map $\bar{\nabla}\varphi$. 
	 Each $B_y$ carries a measure $\nu$ induced from the Lebsgue measure on $\Delta_y^0$. Since Thm. \ref{thm:globalrealMA} applies to all sufficiently irrational $y$ simultaneously, \emph{a fortiori} the measure $\nu$ is  independent of $y$ under the canonical identification of $B_y$.  We can thus think of $\bar{\nabla}\phi$ as a \emph{measure preserving map}  between two probability measure spaces $(Sk(X),\mu)$ and $(B, \nu)$ independent of $y$.
	
	It is a curious question if one can give an a priori construction of $(B,\nu)$ that is manifestly independent of the choice of the sufficiently irrational $y\in Sk(X)$. The proposal in \cite[Section 6.3]{Litheta} suggests that $(B,\nu)$ is the \emph{normalised Lebesgue measure on the essential skeleton corresponding to the mirror family of Calabi-Yau manifolds}. This would be a very appealing metric version of mirror symmetry.
	\end{rmk}


\begin{thebibliography}{7}
		
		
		
		
		
		


			
			
			
			
			
			\bibitem{Hultgren2}
			
			Rolf Andreasson, Jakob Hultgren.
			Solvability of Monge-Amp\`ere equations and tropical affine structures on reflexive polytopes.  	arXiv:2303.05276.
			
			
			
			
			\bibitem{BlumLiu}  Blum, Harold; Liu, Yuchen. Valuative independence for Calabi--Yau varieties. https://arxiv.org/abs/2604.27890
			
			
			
			
			\bibitem{Boucksomsemipositive} 
			Boucksom, S.; Favre, C.; Jonsson, M. Singular semipositive metrics in non-Archimedean geometry. J. Algebraic Geom. 25 (2016), no. 1, 77--139. 
			
			
			\bibitem{Boucksom1} 
			Boucksom, S.; Jonsson, M. Tropical and non-Archimedean limits of degenerating families of volume forms. J. Éc. polytech. Math. 4 (2017), 87--139. 
			
			
			
			
			
			
			
			\bibitem{Boucksom} 
			Boucksom, S.; Favre, C.; Jonsson, M. Solution to a non-Archimedean Monge-Ampère equation. J. Amer. Math. Soc. 28 (2015), no. 3, 617--667.
			
			
			\bibitem{Boucksomsurvey}
			Boucksom, S.; Favre, C.; Jonsson, M. The non-Archimedean Monge-Ampère equation. Nonarchimedean and tropical geometry, 31--49, Simons Symp., Springer, [Cham], 2016. 
			
			
			
			
			\bibitem{Boucksomnew1}  Boucksom, S.; Eriksson, D. Spaces of norms, determinant of cohomology and Fekete points in non-Archimedean geometry. 	
			Adv. Math. 378 (2021), Paper No. 107501, 124 pp. 
			
			
			
			\bibitem{BoucksomOkounkov}  Boucksom, Sébastien. Corps d'Okounkov (d'après Okounkov, Lazarsfeld-Mustaţǎ et Kaveh-Khovanskii). (French) [[Okounkov bodies (following Okounkov, Lazarsfeld-Mustaţǎ and Kaveh-Khovanskiĭ)]] Astérisque No. 361, (2014), Exp. No. 1059, vii, 1--41.
			
			
			
			
			
			
			
			
			
			
			
			\bibitem{CLDucros} 
			Chambert-Loir,  A.; Ducros, A. 
			Formes différentielles réelles et courants sur les espaces de Berkovich. 	arXiv:1204.6277.
			
			
			
			\bibitem{Demailly}
			Demailly, Jean-Pierre. Analytic methods in algebraic geometry. Surveys of Modern Mathematics, 1. International Press, Somerville, MA; Higher Education Press, Beijing, 2012. viii+231 pp. ISBN: 978-1-57146-234-3
			
			
			
			
			\bibitem{Travis}  M-W. Cheung, T. Magee, T. Mandel, G. Muller. Valuative independence and cluster theta reciprocity. https://arxiv.org/abs/2505.09585.
			
			
			
			
			
			
			
			
			\bibitem{EGZ} 
			Eyssidieux, P.; Guedj, V.; Zeriahi, A. Singular Kähler-Einstein metrics. J. Amer. Math. Soc. 22 (2009), no. 3, 607--639.
			
			
			
			
			\bibitem{Favre}
			Favre, C. 
			Degeneration of endomorphisms of the complex projective space in the hybrid space. (English summary)
			J. Inst. Math. Jussieu 19 (2020), no. 4, 1141–1183. 
			
			
			
			\bibitem{GotoOdaka}
			
			Keita Goto, Yuji Odaka. Special Lagrangian fibrations, Berkovich retraction, and crystallographic groups.  	arXiv:2206.14474.
			
			
			
			
	
			
			
			\bibitem{Gross} 
			Gross, M. Mirror symmetry and the Strominger-Yau-Zaslow conjecture. Current developments in mathematics 2012, 133--191, Int. Press, Somerville, MA, 2013. 
			
			
			
			
			
			\bibitem{Grosstheta1} 
			
			Gross, Mark; Siebert, Bernd. Theta functions and mirror symmetry. Surveys in differential geometry 2016. Advances in geometry and mathematical physics, 95--138, Surv. Differ. Geom., 21, Int. Press, Somerville, MA, 2016. 
			
			
			\bibitem{Grosstheta2} 
			Gross, Mark; Hacking, Paul; Siebert, Bernd. Theta functions on varieties with effective anti-canonical class. Mem. Amer. Math. Soc. 278 (2022), no. 1367, xii+103 pp. ISBN: 978-1-4704-5297-1; 978-1-4704-7167-5
			
			\bibitem{Grosstheta3} 
			Gross, Mark; Siebert, Bernd. The canonical wall structure and intrinsic mirror symmetry. Invent. Math. 229 (2022), no. 3, 1101--1202.
			
			
			
			
			
		\bibitem{GrossWilson}
		Gross, M.; Wilson, P. M. H. Large complex structure limits of $K3$ surfaces. J. Differential Geom. 55 (2000), no. 3, 475--546.
			
			

			
			
			
			
			
			
			\bibitem{Hultgren}
			
			
			
			
			Hultgren, Jakob; Jonsson, Mattias; Mazzon, Enrica; McCleerey, Nicholas. Tropical and non-Archimedean Monge-Ampère equations for a class of Calabi-Yau hypersurfaces. Adv. Math. 439 (2024), Paper No. 109494, 42 pp.
			
			
			
			
	\bibitem{HultgrenKhalid} J. Hultgren, M.S. Khalid, in preparation.
			
			
			\bibitem{LiFermat} 
			
			
			Li, Y. Strominger-Yau-Zaslow conjecture for Calabi-Yau hypersurfaces in the Fermat family. Acta Math. 229 (2022), no. 1, 1--53. 
			
			
			
			
			
			
			
			\bibitem{LiNA} Li, Y. Metric SYZ conjecture and non-Archimedean geometry. Duke Math. J. 172 (2023), no. 17, 3227--3255.
			
			

			
			
			
			\bibitem{LiFano}  Li, Y. Metric SYZ conjecture for certain toric Fano hypersurfaces. Camb. J. Math. 12 (2024), no. 1, 223--252.
			
			
			\bibitem{Liintermediate} 
			
			

	
Li, Y. Intermediate complex structure limit for Calabi-Yau metrics. Invent. Math. 240 (2025), no. 2, 459--496.


\bibitem{LiSYZsurvey}
			
Li, Y. Survey on the metric SYZ conjecture and non-Archimedean geometry. Internat. J. Modern Phys. A 37 (2022), no. 17, Paper No. 2230009, 44 pp.
			
			
			
					\bibitem{Litheta}
					
				Li, Y.	Degeneration of Calabi-Yau metrics and canonical basis. https://arxiv.org/abs/2505.11087
				
				
			
						\bibitem{NicaiseXu}
			Nicaise, J.; Xu, C. The essential skeleton of a degeneration of algebraic varieties. Amer. J. Math. 138 (2016), no. 6, 1645--1667.
			
			
			
			
			
			
			
			\bibitem{PilleSchneider}
			Pille-Schneider, L. Hybrid toric varieties and the non-archimedean SYZ fibration on Calabi-Yau hypersurfaces.  	arXiv:2210.05578. 
			
			
			\bibitem{PilleSchneider2}
			
			
			Pille-Schneider, Léonard. Global pluripotential theory on hybrid spaces. J. Éc. polytech. Math. 10 (2023), 601--658.
			
			
			
			
			
			
			\bibitem{SYZ} 
			Strominger, A.; Yau, S-T.; Zaslow, E. Mirror symmetry is $T$-duality.
			Nucl.Phys.B479:243-259,1996.
			
			
			\bibitem{Vilsmeier} 
			

			Vilsmeier, Christian. A comparison of the real and non-archimedean Monge-Ampère operator. Math. Z. 297 (2021), no. 1-2, 633--668. 
			
			
			
			\bibitem{WittNystrom} 
			Witt Nyström, David. Transforming metrics on a line bundle to the Okounkov body. Ann. Sci. Éc. Norm. Supér. (4) 47 (2014), no. 6, 1111--1161.
			
			
			
			
			
			
		
			
		
		
	\end{thebibliography}
\end{document}